\documentclass[11pt, reqno]{amsart}
\usepackage{amssymb,amsmath,epsfig,mathrsfs, enumerate, xparse, mathtools}
\usepackage[pagewise]{lineno}
\usepackage{graphicx}\usepackage[normalem]{ulem}
\usepackage{fancyhdr}
\usepackage[margin=2.5cm]{geometry}
\usepackage{hyperref}
\hypersetup{
 colorlinks   = true,
 urlcolor     = blue,
 linkcolor    = blue,
 citecolor   = red ,
 bookmarksopen=true
}
 \usepackage[usenames]{color}

\theoremstyle{plain}
	\newcommand{\sn}{\mathbb{S}^{n-1}}

\newtheorem{thrm}{Theorem}[section]
\newtheorem{lemma}[thrm]{Lemma}
\newtheorem{prop}[thrm]{Proposition}

\newtheorem{rmrk}[thrm]{Remark}
\newtheorem{dfn}[thrm]{Definition}

\begin{document}
\newcommand{\blue}[1]{\textcolor{blue}{#1}}
\newcommand{\SL}{\mathcal L^{1,p}( D)}
\newcommand{\Lp}{L^p( Dega)}
\newcommand{\CO}{C^\infty_0( \Omega)}
\newcommand{\Rn}{\mathbb R^n}
\newcommand{\Rm}{\mathbb R^m}
\newcommand{\R}{\mathbb R}
\newcommand{\Om}{\Omega}
\newcommand{\Hn}{\mathbb H^n}
\newcommand{\aB}{\alpha B}
\newcommand{\eps}{\ve}
\newcommand{\BVX}{BV_X(\Omega)}
\newcommand{\p}{\partial}
\newcommand{\IO}{\int_\Omega}
\newcommand{\bG}{\boldsymbol{G}}
\newcommand{\bg}{\mathfrak g}
\newcommand{\bz}{\mathfrak z}
\newcommand{\bv}{\mathfrak v}
\newcommand{\Bux}{\mbox{Box}}
\newcommand{\e}{\ve}
\newcommand{\X}{\mathcal X}
\newcommand{\Y}{\mathcal Y}
\newcommand{\W}{\mathcal W}
\newcommand{\la}{\lambda}
\newcommand{\vf}{\varphi}
\newcommand{\rhh}{|\nabla_H \rho|}
\newcommand{\Ba}{\mathcal{B}_\gamma}
\newcommand{\Za}{Z_\beta}
\newcommand{\ra}{\rho_\beta}
\newcommand{\na}{\nabla_\beta}
\newcommand{\vt}{\vartheta}
\newcommand{\La}{\mathcal{L}}

\numberwithin{equation}{section}

\newcommand{\RN} {\mathbb{R}^N}
\newcommand{\Sob}{S^{1,p}(\Omega)}
\newcommand{\Dxk}{\frac{\partial}{\partial x_k}}
\newcommand{\Co}{C^\infty_0(\Omega)}
\newcommand{\Je}{J_\ve}
\newcommand{\beq}{\begin{equation}}
\newcommand{\bea}[1]{\begin{array}{#1} }
\newcommand{\Dt}{\partial_{t}}
\newcommand{\eeq}{ \end{equation}}
\newcommand{\ea}{ \end{array}}
\newcommand{\eh}{\ve h}
\newcommand{\Dxi}{\frac{\partial}{\partial x_{i}}}
\newcommand{\Dyi}{\frac{\partial}{\partial y_{i}}}
\newcommand{\aBa}{(\alpha+1)B}
\newcommand{\GF}{\psi^{1+\frac{1}{2\alpha}}}
\newcommand{\GS}{\psi^{\frac12}}
\newcommand{\HFF}{\frac{\psi}{\rho}}
\newcommand{\HSS}{\frac{\psi}{\rho}}
\newcommand{\HFS}{\rho\psi^{\frac12-\frac{1}{2\alpha}}}
\newcommand{\HSF}{\frac{\psi^{\frac32+\frac{1}{2\alpha}}}{\rho}}
\newcommand{\AF}{\rho}
\newcommand{\AR}{\rho{\psi}^{\frac{1}{2}+\frac{1}{2\alpha}}}
\newcommand{\PF}{\alpha\frac{\psi}{|x|}}
\newcommand{\PS}{\alpha\frac{\psi}{\rho}}
\newcommand{\ds}{\displaystyle}
\newcommand{\Zt}{{\mathcal Z}^{t}}
\newcommand{\XPSI}{2\alpha\psi \begin{pmatrix} \frac{x}{|x|^2}\\ 0 \end{pmatrix} - 2\alpha\frac{{\psi}^2}{\rho^2}\begin{pmatrix} x \\ (\alpha +1)|x|^{-\alpha}y \end{pmatrix}}
\newcommand{\Z}{ \begin{pmatrix} x \\ (\alpha + 1)|x|^{-\alpha}y \end{pmatrix} }
\newcommand{\ZZ}{ \begin{pmatrix} xx^{t} & (\alpha + 1)|x|^{-\alpha}x y^{t}\\
     (\alpha + 1)|x|^{-\alpha}x^{t} y &   (\alpha + 1)^2  |x|^{-2\alpha}yy^{t}\end{pmatrix}}
\newcommand{\norm}[1]{\lVert#1 \rVert}
\newcommand{\ve}{\varepsilon}
\newcommand{\ee}{e^{-\frac{\beta}{2} (\log \sigma)^{2}}}
\newcommand{\eer}{e^{\beta (\log \sigma_{1}(r_1/2, r_1, T))^2}}
\newcommand{\ees}{e^{\beta (\log (\sigma_{1}(r_2, r_3, T))^2}}
\newcommand{\suplog}{e^{\beta (\log (\sup\limits_{(r_1, r_2)}\sigma(x,t)))^2}}
\newcommand{\eesr}{e^{\beta \bigg[(\log (\sigma_1(r_1/2, r_1, T)))^2-(\log (\sigma_2(r_1, {M r_2}/{N}, T)))^2\bigg]}}
\newcommand{\eelog}{e^{\beta \bigg[(\log (\sigma_1(r_2, r_3, T)))^2-(\log (\sigma_2(r_1, {M r_2}/{N}, T)))^2\bigg]}}
\title{}
\title[Carleman estimates and unique continuation]{ strong unique continuation  for  variable coefficient parabolic operators with Hardy type potential}
\author{Agnid Banerjee}
	\address{Tata Institute of Fundamental Research\\
		Centre For Applicable Mathematics \\ Bangalore-560065, India}\email[Agnid Banerjee]{agnid@tifrbng.res.in}
	\author{Pritam Ganguly}
	\address{Department of Mathematics\\ Indian Institute of Science\\ Bangalore-560012, India}\email[Pritam Ganguly]{pritam1995.pg@gmail.com}
	
	\author{Abhishek Ghosh}
\address{Tata Institute of Fundamental Research\\
		Centre For Applicable Mathematics \\ Bangalore-560065, India}\email[Abhishek Ghosh]{abhi21@tifrbng.res.in}

\thanks{A.B is supported in part by  Department of Atomic Energy,  Government of India, under
project no.  12-R \& D-TFR-5.01-0520.}
\thanks{P.G is supported in part by the J.C. Bose fellowship of Prof. S. Thangavelu, grant no. DSTO-2036.}


%
%
%

\maketitle

\tableofcontents

\begin{abstract}
In this paper, we prove the strong unique continuation property at the origin for  solutions of the following scaling critical  parabolic differential inequality 
		\[
			|\operatorname{div} (A(x,t) \nabla u) - u_t| \leq \frac{M}{|x|^{2}} |u|,\ \ \ \
		\]
		 where the coefficient matrix $A$ is Lipschitz continuous in $x$ and $t$. Our main result sharpens a previous one of Vessella concerned with the subcritical case as well as extends a recent  result of one of us with Garofalo and Manna  for the heat operator.
		
\end{abstract}

\section{Introduction}
Roughly speaking, a differential operator $P$ is said to have the strong unique continuation property (sucp) if a solution $u$ to $Pu=0$  vanishes to infinite order at a point  in a connected domain, then $u \equiv 0$. The study of unique continuation problems has  its root in an old paper of Carleman written in 1939, in which, Carleman studied the unique continuation problem associated with the Schr\"odinger operator $H=-\Delta+V$ with bounded potential $V$ in $\R^2$ by means of  a weighted  apriori  estimate. Such an estimate was subsequently extended to variable coefficient elliptic operators with Lipschitz principal part in \cite{AKS}, see also \cite{A} where the case of $C^{2}$ coefficients was earlier considered.   We refer to some of  the later prominent works \cite{ ChS, EF, Ho, JK, KT}   in this subject  and an interested reader can find other references therein. An alternate method based on the almost monotonicity of a generalization of the frequency function, first introduced by Almgren in \cite{Al} came up in the work of Garofalo and Lin in \cite{GL}.  Using this approach, they were able to obtain new quantitative uniqueness information for the solutions to divergence form elliptic equations with Lipschitz coefficients which in particular encompassed the results  in \cite{AKS}. We recall that unique continuation fails in general when the coefficients are only H\"older continuous, see for instance \cite{Pl, Mi}.

 In this paper, we obtain strong unique continuation property at the origin for solutions to the differential inequality
\begin{align}
    |\operatorname{div} (A(x,t) \nabla u) - u_t| \leq \frac{M}{|x|^{2}} |u|,
    \label{main-diff-ineq}
\end{align}
where $M>0$ and   the matrix valued function $A$ satisfies the following uniform ellipticity and the Lipschitz growth condition
\begin{equation}\label{assump}
\begin{cases}
\Lambda^{-1} \mathbb I\leq A(x,t) \leq \Lambda \mathbb I\ \text{for some $\Lambda >1$},
\\
|A(x,t) - A(y,s)| \leq K(|x-y| + |t-s|).
\end{cases}
\end{equation}

 It is well-known that  the inverse-square potential $V(x) = \frac{M}{|x|^2}$ represents a critical scaling threshold in quantum mechanics \cite{BG}, and it is equally known that its singularity is the limiting case for the sucp for the differential inequality $|\Delta u| \le \frac{M}{|x|^m} |u|$, see the counterexample in \cite{GL}. Such potential fails to be in $L^{n/2}_{loc}$, and in general does not have small $L^{n/2,\infty}$ seminorm, thus in the context of the Laplacian the sucp cannot be treated by the celebrated result of Jerison and Kenig in \cite{JK} or the subsequent improvement by Stein in the appendix to the same paper. 
We recall that, in the time-independent case of the Laplacian, the sucp for the unrestricted inverse square potential was proved by Pan in \cite{Pa}. This was later extended to Lipschitz principal part  by Regbaoui  in \cite{Reg}. We refer to \cite{LNW} for quantitative results in this setting and also to  \cite{BM} for  a similar result in the subelliptic framework of a subclass of  Gr\"ushin type operators.

Now in the parabolic setting, in \cite{V} (see also \cite{EV}) Vessella proved a general sucp result for sub-critical parabolic equations of the type 
\begin{equation}\label{v}
|\operatorname{div} (A(x,t) \nabla u) - u_t| \leq \frac{M}{|x|^{2-\delta}} |u|,\ \ \ \ \delta>0,
\end{equation}
under the same  Lipschitz regularity assumptions on the principal part $A(x,t)$ as in \eqref{assump} above.   Recently,  one of us with Garofalo and Manna in \cite{BGM} proved the sucp for solutions to the  scaling critical differential inequality \eqref{main-diff-ineq} in the case when $A=\mathbb{I}$. This was done  by means of an improved Carleman estimate in the case of the heat operator $\Delta -\partial_t$ in space-time cylinders.  Similar to the time-independent case in \cite{Pa} and \cite{Reg}, the proof of such a scaling critical Carleman estimate in \cite{BGM} exploited  the spectral gap on $\sn$ combined with another delicate apriori estimate  which was an important new feature in the parabolic  setting.  The purpose of this work  is to  extend the sucp result in \cite{BGM} to parabolic operators with Lipschitz principal part. Our main result Theorem \ref{thm1-bgg} constitutes the parabolic counterpart of  the one due to Regbaoui in \cite{Reg}.
\medskip

The following  are the key steps in the proof of our main result.

\medskip

\emph{Step1:}  Inspired by ideas  in \cite{Reg},  we first establish a sharpened version of the scaling critical Carleman estimate in \cite{BGM}  ( see \eqref{est1} below). This has required some delicate  reworking of the ideas in \cite{BGM}. 
Using such an estimate combined with a suitable change of variable in the time-dependent setting, we show that  if a solution $u$ to \eqref{main-diff-ineq} vanishes to infinite order at the origin  in the sense of \eqref{vp} below, then it decays exponentially. See Proposition \ref{expdecay} below.

\medskip

\emph{Step 2:} Then  by means of a parabolic generalization of a Carleman estimate due to  Regbaoui in \cite{Reg}( see \eqref{har1-BGG} below), we then show    that non-trivial solutions to  \eqref{main-diff-ineq} in fact decay less than exponentially which  then lead to a contradiction and  thus establishes our main result Theorem \ref{thm1-bgg}.  We mention that the proof of the corresponding estimate in \cite{Reg} uses in a crucial way the polar decomposition of the frozen constant coefficient operator combined with a  H\"ormander type commutator estimate.   Our proof in the parabolic case instead is based on  a suitable adaptation of a  Rellich type identity and is partly inspired by some ideas in  the recent work \cite{BMa} where a similar estimate has been established for  Gr\"ushin type operators. However as the reader will see, our proof entails some  subtle modifications  in the time dependent setting.  This is in fact one of the key novelties of our work.

The paper is organized as follows. In Section \ref{s:p}, we introduce some relevant notions, gather some known results and then state our main strong unique continuation result.  In Section \ref{s:main}, we prove our main result.

\section{Preliminaries}\label{s:p}
Given $r>0$ we  denote by $B_r(x_0)$ the Euclidean ball centred at  $x_0 \in \R^n$ with radius $r$. When $x_0=0$, we will use the simpler notation $B_r$.   A generic point in space time $\Rn \times (0, \infty)$ will be denoted by $(x,t)$. For notational convenience, $\nabla f$ and  $\operatorname{div}\ f$ will respectively refer to the quantities  $\nabla_x f$ and $ \operatorname{div}_x f$ of a given function $f$.   The partial derivative in $t$ will be denoted by $\p_t f$ and also by $f_t$. We indicate with $C_0^{\infty}(\Omega)$ the set of compactly supported smooth functions in the region $\Omega$  in space-time. By $H^{2,1}_{loc}(\Omega)$ we refer to the parabolic Sobolev class of functions $f\in L^2_{loc}(\Om)$ for which the weak derivatives $\nabla f, \nabla^2 f $ and $\partial_t f$ belong to $L^{2}_{loc}(\Om)$. For a point $x\in \Rn\setminus\{0\}$, we will routinely adopt the notation $r = r(x) = |x|$ and $\omega = \frac{x}{r}\in \sn$, so that $x = r \omega$. The radial derivative of a function $v$ is 
$v_r = <\nabla v,\frac{x}{|x|}>$.

The relevant notion of vanishing to infinite order is as follows.

\begin{dfn}
We say a function $u$ parabolically vanishes to infinite order if for all $k>0$, we have
\begin{equation}\label{vp}
\int_{B_r \times (0, T)} u^2 = O(r^k),\, \text{as}\,\, r \to 0.
\end{equation}    
\end{dfn}

\subsection*{Statement of the main result}
We now  state our main result.
\begin{thrm}\label{thm1-bgg}
Suppose that for some $M, R, T>0$ the function $u \in H^{2,1}_{loc}$ be a solution in $B_R \times (0, T)$ to the  differential inequality \eqref{main-diff-ineq}. If $u$ parabolically vanishes to infinite order at the origin in the sense of \eqref{vp}, then $u \equiv 0$ in $B_R \times (0, T)$. 
\end{thrm}

\begin{rmrk}\label{H2}
We  remark that even if we assume that  $u \in H^{1}_{loc}$ ( i.e. $u, \nabla u \in L^2_{loc}$)  is a weak solution to \eqref{main-diff-ineq} satisfying parabolic vanishing property \eqref{vp},  then  it follows that $u \in H^{2,1}_{loc}$ and moreover  we have that $\nabla u, \nabla^2 u$ and  $u_t$  vanish to infinite order in the sense of \eqref{vp} above. This is seen as follows. Recall the following estimate from \cite[Chapter 6]{Liberman}
\begin{align}
\label{H2-estminate}
\int_{Q_{2r} (0, t_0)\setminus Q_{r}(0, t_0)} \left(u_{t}^{2}+|\nabla^2 u|^2\right)\lesssim \frac{1}{r^4} \int_{Q_{4r} (0,t_0) \setminus Q_{r} (0, t_0)} u^2,    
\end{align}
for any $t_0 \in (-3T/4, 3T/4)$.
 Since $u$ vanishes parabolically to infinite order, we conclude
\begin{align}
\label{H2-estminate-1}
\int_{Q_{2r}(0, t_0)\setminus Q_{r} (0, t_0)} \left(u_{t}^{2}+|\nabla^2 u|^2\right)\lesssim r^k,  
\end{align}
for all $k \geq 1$.
From \eqref{H2-estminate-1} it follows
\begin{equation}\label{h22}
\int_{B_{2r} \setminus B_r \times (-3T/4, 3T/4)} \left(  u_t^2 + |\nabla^2 u|^2 \right) \lesssim r^k.\end{equation}
 Now for any $r<1$, recursively applying \eqref{h22} with $r_{i}=\frac{r}{2^i}$ for all $i\geq 1$, then summing over the annular regions we obtain the following
\begin{align*}
 \int_{B_{r}\times (-3T/4, 3T/4)} \left(u_{t}^{2}+|\nabla^2 u|^2\right)= O(r^k),   
\end{align*}
for all $k\geq 1$. Moreover  by applying Moser type estimate  in annular regions, it also follows that  a solution $u$ to \eqref{main-diff-ineq} satisfying \eqref{vp} also decays to infinite order in the $L^{\infty}$ norm in the following sense
\begin{equation}\label{vp1}
||u||_{L^{\infty}(B_r \times (-3T/4, 3T/4)}  = O(r^k),
\end{equation}
for all $k \geq 1$. Therefore one only requires $u \in H^{1}_{loc}$ in  Theorem \ref{thm1-bgg} above.
\end{rmrk}

We close this section with the  following  elementary algebraic  inequality that will be needed  in the proof of  Theorem \ref{thm2} below.
\begin{lemma}\label{alineq}
Given $a, b \in \R^n$ and $\delta>0$, the following inequality holds
\begin{equation}\label{ai1}
(1+\delta) |a|^2 + |b|^2 + 2\langle a, b\rangle \geq C(\delta) ( |a|^2 + |b|^2),
\end{equation}
for some $C(\delta) >0$. 
\end{lemma}
\begin{proof}
By an application of Cauchy-Schwartz inequality, we have
\begin{equation}\label{ai0}
2\langle a, b\rangle \geq  - \left(1 +\frac{\delta}{2}\right) |a|^2 - \frac{1}{1+\frac{\delta}{2}} |b|^2.
\end{equation}
Using \eqref{ai0} we find
\begin{align}
&(1+\delta) |a|^2 + |b|^2 + 2\langle a, b\rangle \geq \frac{\delta}{2} |a|^2 + \frac{\delta}{2(1+\frac{\delta}{2})} |b|^2\\
& \geq \frac{\delta}{2(1+\frac{\delta}{2})}  (|a|^2 +|b|^2).\notag
\end{align}

Thus the  inequality \eqref{ai1} is seen to hold with $C(\delta)=  \frac{\delta}{2(1+ \frac{\delta}{2}) }.$

\end{proof}
	
	\section{ Proof of the main result}\label{s:main}
	\subsection{Carleman estimate I}
	We first state and prove our first Carleman estimate with singular weights which can be regarded as a certain sharpened version of Theorem 1.1 in \cite{BGM}. The difference of our estimate as in \eqref{est1} below  from that in \cite[Theorem 1.1]{BGM} is the incorporation of higher order terms. This   is needed  in order to  absorb certain error terms that arises due to the perturbation of the principal part.  Such an estimate  in conjunction with the infinite  vanishing property \eqref{vp} leads to the exponential decay of solutions to \eqref{main-diff-ineq}.  As the reader will see, this requires  some subtle  reworking of the proof of Theorem 1.1 in \cite{BGM}.
	\begin{thrm} \label{thm2}
		Let $R<1$ and let  $u \in C_0^{\infty}((B_R \setminus \{0\}) \times (0,T))$. For all $\alpha$ sufficiently large of the form $\alpha= k +\frac{n+1}{2}$ with $k \in \mathbb N$, and every $0<\ve<<1$ very small, we have 
		\begin{align}\label{est1}
			&\alpha\int_{B_R \times (0,T)} |x|^{-2\alpha-4} e^{2\alpha |x|^{\ve}} \, u^2  dxdt  +\alpha^3\int_{B_R \times (0, T)} |x|^{-2\alpha- 4+\ve} e^{2\alpha |x|^{\ve}}  u^2  \\  
			&+\frac{1}{\alpha}\int_{B_R \times (0,T)} |x|^{-2\alpha-2} e^{2\alpha |x|^{\ve}} \, |\nabla u|^2  dxdt+\frac{1}{\alpha^3}\int_{B_R \times (0,T)} |x|^{-2\alpha} e^{2\alpha |x|^{\ve}} \, |\nabla^2u|^2  dxdt\nonumber\\
			&\leq C\int_{B_R \times (0, T)} |x|^{-2\alpha} e^{2\alpha |x|^{\ve}} (\Delta u - \partial_t u)^2   dx dt \nonumber
		\end{align}
		where $C= C(\ve, n)>0$.
	\end{thrm}

 We  record the following simple lemma which can be regarded as an integration by parts formula for radial derivatives,  will be repeatedly used in our analysis. See for instance Lemma 2.1  \cite{BGM}. 	\begin{lemma}\label{L:radial}
		Let $f\in C^\infty_0(\Rn\setminus\{0\})$, $g\in C^\infty(\Rn\setminus\{0\})$, then
		\[
		\int_{\Rn} f_r g dx = - \int_{\Rn} f g_r dx - (n-1) \int_{\Rn} r^{-1} f g dx.
		\]
	\end{lemma}
	
		\begin{proof}[Proof of Theorem \ref{thm2}] 
		 Before we proceed with the proof, in order to avoid any notational confusion, we declare that, in what follows, the domain of all the integrations will be the parabolic cylinder $B_R \times (0,T)$ (or, for that matter, the whole of $\Rn\times \R$, in view of the support property of the integrands), but this will not be explicitly indicated. Nor  we will explicitly write the measure $dx dt$ in any of the integrals involved.  Let $r=|x|$ and  also let $v=r^{-\beta} e^{\alpha r^\ve}  u$, where as in \cite{BGM}, $\beta$ is chosen in the following way
		 \begin{equation}\label{rel}
		 2 \beta - 2 \alpha -4 =-n.
		 \end{equation}

		  Thus  $u= r^{\beta} e^{-\alpha r^\ve} v$ and an easy calculation yields
		\[
		\Delta(r^{\beta} e^{-\alpha r^\ve})=\left(\alpha^2 \ve^2 r^{\beta+2\ve-2} + \beta(\beta+n-2)r^{\beta-2}- \alpha\ve\left((2\beta+\ve+ n-2)\right) r^{\beta+\ve-2}\right) e^{-\alpha r^{\ve}}
		\]
		leading to
		\begin{align*}
			\Delta u & =r^{\beta} e^{-\alpha r^\ve}  \Delta v+\left(\alpha^2 \ve^2 r^{\beta+2\ve-2} + \beta(\beta+n-2)r^{\beta-2}- \alpha\ve\left((2\beta+\ve+ n-2)\right) r^{\beta+\ve-2}\right) e^{-\alpha r^{\ve}} v
			\\
			& + \left(2\beta r^{\beta-2} -2 \ve \alpha r^{\beta+\ve-2}\right) e^{-\alpha r^{\ve}} <x, \nabla v>.
		\end{align*}
		Since $\Delta v(x,t) =v_{rr}(r\omega,t)+\frac{n-1}{r} v_r(r\omega,t)+\frac{1}{r^2} \Delta_{\sn} v(r\omega,t),$
		where $\omega\in \sn$ and $\Delta_{\sn}$ denotes the Laplacian on $\mathbb{S}^{n-1}$,  we obtain
		\begin{align}\label{comp1}
			& \Delta u -\partial_tu = r^\beta e^{-\alpha r^{\ve}} \bigg[\left(\alpha^2 \ve^2 r^{2\ve-2} + \beta(\beta+n-2)r^{-2}- \alpha\ve\left(2\beta+\ve+ n-2\right) r^{\ve-2}\right) v
			\\
			& + \left((2\beta+n-1) r^{-1} -2 \alpha \ve r^{\ve-1}\right) v_r + r^{-2} \Delta_{\sn} v+ v_{rr} - v_t\bigg].\notag
		\end{align}
		We now apply the algebraic inequality $(a+b)^2 \geq a^2 + 2ab$, with 
		\[
		a= r^{\beta -2} e^{-\alpha r^\ve} \left( \beta(\beta+n-2) v + \Delta_{\sn} v + (2\beta+ n-1)r v_r  +r^2v_{rr}\right),
		\] 
		and 
		\[
		b = r^\beta e^{-\alpha r^\ve}  \left(\alpha^2 \ve^2 r^{2\ve-2} v - \alpha\ve(2\beta+\ve+ n-2) r^{\ve-2} v
		-2 \alpha \ve r^{\ve-1} v_r - v_t\right),
		\]
		obtaining 
		\begin{equation}
			\label{rt12nd}
			\int r^{-2\alpha}  e^{2\alpha r^\ve}(\Delta u - \partial_t u)^2  \geq I+II+III+IV
		\end{equation}
	where \begin{equation}
		\label{defI}
		I:= \int r^{-2\alpha+2\beta-4} (\beta(\beta+n-2) v + \Delta_{\sn}v )^2,
	\end{equation}
	\begin{equation}
	\label{defII}
II:=	2\beta(\beta+n-2)(2\beta+n-1)\int r^{-2\alpha+2\beta-3} v v_r,
	\end{equation}
		\begin{align}
			\label{defIII}
			III:&=(2\beta+n-1)^2 \int r^{-2\alpha + 2\beta -2} v_r^2 +\int r^{-2\alpha+2\beta}v_{rr}^2
			\\
			& + 2(2\beta+n-1)\int r^{-2\alpha+2\beta-3} v_r \Delta_{\sn}v
			\notag
			\\
			& +2\beta(\beta+n-2)\int r^{-2\alpha+2\beta-2}v v_{rr} \notag
			+2\int r^{-2\alpha+2\beta-2}v_{rr}\Delta_{\sn}v
			\notag\\
			& + 2(2\beta+n-1) \int r^{-2\alpha + 2\beta -1} v_r v_{rr},~ \text{and}\notag 
		\end{align} 
	\begin{align}
		\label{defIV}
		&IV:= 2 \alpha^2 \ve^2 \beta(\beta+n-2)   \int r^{-2\alpha +2\beta +2\ve -4} v^2-  2 \alpha \ve \beta(\beta+n-2)  (2 \beta+ \ve+n-2)  \int r^{-2\alpha + 2\beta +\ve -4} v^2 
		\\
		&- 4  \alpha \ve \beta(\beta+n-2) \int r^{-2\alpha+2\beta +\ve -3} vv_r+2 \alpha^2\ve^2  \int r^{-2\alpha +2\beta +2\ve -4} v \Delta_{\sn} v
		\notag\\
		&-  2 \alpha\ve\left(2\beta+\ve+ n-2\right) \int r^{-2\alpha + 2\beta + \ve -4} v \Delta_{\sn} v- 4  \alpha  \ve \int r^{-2\alpha +2\beta+ \ve-3 } v_r \Delta_{\sn} v    
		\notag\\
		& - 2\alpha \ve (2\beta+n-1) (2 \beta+ \ve + n-2) \int r^{-2\alpha + 2\beta+\ve -3} vv_r 
		+ 2 \alpha^2 \ve^2 (2\beta+n-1)  \int r^{-2\alpha+2\beta+2\ve -3} vv_r 
		\notag\\
		& - 4 \alpha \ve( 2\beta+n-1)  \int r^{-2 \alpha +2\beta + \ve -2} v_r^2 +2\alpha^2\ve^2\int r^{-2\alpha+2\beta+2\ve-2}vv_{rr}
		\notag\\
		&-2\alpha\ve(2\beta+\ve+n-2)\int r^{-2\alpha+2\beta+\ve-2}vv_{rr}-4\alpha\ve \int r^{-2\alpha+2\beta+\ve-1} v_rv_{rr}\notag \\
		 &-2\beta(\beta+n-2)\int r^{-2\alpha+2\beta-2} v v_t-2\int r^{-2\alpha+2\beta-2}v_t\Delta_{\sn}v \notag\\
			& -2 (2\beta+n-1) \int r^{-2\alpha+2\beta-1} v_t v_r-2\int r^{-2\alpha+2\beta} v_tv_{rr}. \notag
 	\end{align}
	 Now we proceed to estimate  each integral listed above. For the convenience of the reader we have grouped the integrals in the above way. More precisely, I, II, III cover all the terms coming from `$a^2$' while $IV$ contains all the terms that comes from `$2ab$'.  Also note that  in our case, the choice of   $a$ and $b$ are different from that in \cite{BGM}. With   $\beta$ as in \eqref{rel}, we find  that the integral $II$ vanishes by an application of Lemma \ref{L:radial} analogous to that in \cite{BGM}.  
	 
	 \medskip

\noindent\textbf{{\it\underline{Estimate for $III$:}}} 

We now turn our attention to the terms of $III.$ 
		From \cite[(2.6), (2.11) and (2.12)]{BGM}, it follows
		\begin{align}\label{zer0} 
			& 2 \int r^{-2\alpha+2\beta-3} v_r \Delta_{\sn}v =0,
		\end{align}	\[
	\int r^{-2\alpha+2\beta-2} v v_{rr}  =- \int r^{-n+2} v_{r}^2,
	\]
and
\[\int r^{-2\alpha + 2\beta -1} v_r v_{rr}=-\int r^{-n+2}v_r^2.\] 	Again another application of Lemma \ref{L:radial} gives
\begin{align} \label{3.9}
	& 2\int r^{-2\alpha+2\beta-2}v_{rr}\Delta_{\sn}v = 2\int r^{-n+2} v_{rr} \Delta_{\sn} v 
	\\
	&= -2 \int r^{-n+2} v_r \Delta_{\sn} v_r - 2 (n-1)\int r^{-n+1} v_r \Delta_{\sn} v    
	\nonumber\\
	&=2\int_{0}^T \int_0^{\infty} r\int_{\mathbb{S}^{n-1}} |\nabla_{\sn} v_r|^2 d\omega dr dt,
	\nonumber
\end{align}
since $\int r^{-n+1} v_r \Delta_{\sn} v=0$. Now we use the notation $\nabla_Tv:=r^{-1}\nabla_{\sn} v$ in view of which, we have $$|\nabla v|^2=v_r^2+|\nabla_Tv|^2.$$ So, this notation along with \eqref{3.9} yield
\begin{equation}\label{hj}
	2\int r^{-2\alpha+2\beta-2}v_{rr}\Delta_{\sn}v=2\int r^{-n+4}|\nabla_Tv_r|^2.
\end{equation} 

Thus from \eqref{defIII}-\eqref{hj} it follows 
\begin{align}
	\label{est:III}
	III& \geq 2\alpha^2 \int r^{-n+2} v_{r}^2+\int r^{-n+4}v_{rr}^2+2\int r^{-n+4}|\nabla_Tv_r|^2.\end{align}
We now proceed to estimate $IV.$

\noindent\textbf{{\it\underline{Estimate for $IV$:}}}  First we notice  by integrating  by parts  in $t$ as in \cite{BGM} that the following holds	\begin{equation}\label{ut1}
		\int r^{-2\alpha+2\beta-2} v v_t =0,
		\end{equation}
	and		\begin{align}\label{ut2}
			& 2 \int r^{-2\alpha+2\beta-2} v_t \Delta_{\sn}v   = 0.
		\end{align}
		
	Next, as in \cite[(2.7)-(2.9)]{BGM} we have		\begin{equation}\label{sp1}
			- 4  \alpha \ve \int r^{-2\alpha + 2\beta + \ve - 3} v_r \Delta_{\sn} v  = -2  \alpha  \ve^2 \int r^{-n+\ve} |\nabla_{\sn} v|^2,
		\end{equation}
		
		\begin{align}\label{sp2}
			& -  2 \alpha\ve\left(2\beta+\ve+ n-2\right)\int r^{-2\alpha+2\beta + \ve - 4} v \Delta_{\sn} v
			= 2 \alpha\ve\left(2\alpha+\ve+ 2\right)  \int r^{-n+\ve}  |\nabla_{\sn} v|^2,
		\end{align}
		and
		\begin{equation}\label{sp22}
			2\alpha^2 \ve^2  \int r^{-2\alpha+2\beta + 2\ve - 4} v \Delta_{\sn} v=- 2\alpha^2 \ve^2    \int r^{-n+2\ve}  |\nabla_{\sn} v|^2\geq - 2\alpha^2 \ve^2    \int r^{-n+\ve}  |\nabla_{\sn} v|^2
		\end{equation}
	where in the last inequality we have used $r^{\ve}<1$ as $0<r\le R<1.$
	 So, from \eqref{sp1}, \eqref{sp2} and \eqref{sp22} we observe that for $\ve$ sufficiently small (for instance take $0<\ve<\frac1{20}$ )
		\begin{align}\label{sp4}
			&- 4 \ve \alpha  \int r^{-2\alpha +2\beta+ \ve-3 } v_r \Delta_{\sn} v -  2 \alpha\ve\left(2\beta+\ve+ n-2\right) \int r^{-2\alpha + 2\beta + \ve -4} v \Delta_{\sn} v
			\\
			&  +2 \alpha^2\ve^2  \int r^{-2\alpha +2\beta +2\ve -4} v \Delta_{\sn} v\geq (2\alpha^2(2\ve-\ve^2)+4\alpha\ve)   \int r^{-n+\ve} |\nabla_{\sn} v|^2
			\notag\\
			&\hspace*{5.5cm}\geq \frac{39}{10}\alpha^2\ve\int r^{-n+\ve} |\nabla_{\sn} v|^2.\notag
		\end{align}

		Furthermore, it follows from \cite[(2.14)]{BGM} that for all $\alpha$ sufficiently large and $\ve$ small enough, the following inequality holds for some $C>0$ universal
		\begin{align}\label{v2}
			&- 2 \alpha \ve (2\beta+n-1)  (2 \beta+ \ve + n-2) \int r^{-2\alpha + 2\beta+\ve -3} vv_r  + 2 \alpha^2 \ve^2 (2\beta+n-1)  \int r^{-2\alpha+2\beta+2\ve -3} vv_r 
			\\
			&- 4  \alpha \ve \beta(\beta+n-2) \int r^{-2\alpha+2\beta +\ve -3} vv_r-  2 \alpha \ve \beta(\beta+n-2) (2\beta+\ve+n-2)  \int r^{-2\alpha + 2\beta +\ve -4} v^2  
			\notag
			\\
			& 
			+ 2 \alpha^2 \ve^2 \beta(\beta+n-2) \int r^{-2\alpha +2\beta +2\ve -4} v^2
			\geq - C \alpha^4 \ve \int r^{-n+\ve} v^2.  
			\notag
		\end{align}
		Also as in \cite[(2.16)]{BGM} we have for all $\alpha$ large
		\begin{align}\label{tyu}
			& (2\beta+n-1)^2 \int r^{-2\alpha + 2\beta - 2} v_r^2 + 2 (2\beta+n-1)\int r^{-2\alpha + 2\beta -1} v_r v_{rr}
			\\
			& - 4 \alpha \ve (2\beta+n-1) \int r^{-2\alpha + 2 \beta +\ve -2} v_r^2
			\notag\\
			& \ge \big[(2\alpha+3)^2 - 2(2\alpha+3) - \alpha^2\big]\int r^{-n +2} v_r^2 \ge 2 \alpha^2 \int r^{-n +2} v_r^2.
			\notag
		\end{align}
	Now we will handle the integrals involving the term $v_{rr}.$	Once again using  Lemma \ref{L:radial} we have 
	\[\int r^{-2\alpha+2\beta+2\ve-2}vv_{rr}=-\int r^{-n+2\ve+2}v_r^2-(2\ve+1)\int r^{-n+2\ve+1}v_rv.\] Again applying the same lemma to the last term we get 
	\[\int r^{-n+2\ve+1}v_rv=-\ve \int r^{-n+2\ve}v^2\] which yields 
	\begin{equation}
		\label{vrr-10}
		\int r^{-2\alpha+2\beta+2\ve-2}vv_{rr}=-\int r^{-n+2\ve+2}v_r^2+(2\ve+1)\ve\int r^{-n+2\ve}v^2.
	\end{equation}
Similarly, we can check that 
\begin{equation}
	\label{vrr-11}
	\int r^{-2\alpha+2\beta+\ve-2}vv_{rr}=-\int r^{-n+\ve+2}v_r^2+\frac12(\ve+1)\ve\int r^{-n+\ve}v^2,
\end{equation}
and also\begin{equation}
	\label{vrr-12}
	\int r^{-2\alpha+2\beta+\ve-1} v_rv_{rr}=-(\ve+2)\int r^{-n+\ve+1}v_r^2.
\end{equation}
Hence  we have  for all large enough $\alpha$ and small $\ve$ that the following inequality holds
\begin{align}
	&2\alpha^2\ve^2\int r^{-2\alpha+2\beta+2\ve-2}vv_{rr}-2\alpha\ve(2\beta+\ve+n-2)\int r^{-2\alpha+2\beta+\ve-2}vv_{rr}-4\alpha\ve \int r^{-2\alpha+2\beta+\ve-1} v_rv_{rr}\\
	&\geq (2\alpha^2\varepsilon+2\alpha\varepsilon^2+4\alpha\varepsilon)\int r^{-n+2+\varepsilon}v_r^2 - C \alpha^2 \ve^2\int r^{-n+\varepsilon} v^2.\notag  
\end{align}
Finally, we complete estimates of the integrals involving the term $v_t.$ Recall that we have already dealt with the 13th and 14th integrals of $IV$, which turn out to be zero. See \eqref{ut1} and \eqref{ut2} above, So, we are just left with the  15th and 16th integral in this regard. Now the 15th integral can be handled as follows
		\begin{align}\label{10}
			& \left|2(2\beta+n-1) \int r^{-2\alpha+2\beta-1} v_t v_r\right|\le 4\alpha(1+\frac{3}{2\alpha})\int r^{-n+3} |v_t| |v_r|
			\\
			& \le 5\alpha\left(\frac{\alpha}5 \int  r^{-n+2} v_r^2 + \frac{5}{\alpha} \int  r^{-n+4} v_t^2\right) \le \alpha^2 \int  r^{-n+2} v_r^2 + 25 \int  r^{-n+4} v_t^2.
			\notag
		\end{align}
	Also, the 16th integral can be handled as follows.
	  		\begin{align}\label{vt-16}
	  	\left|2\int r^{-2\alpha+2\beta} v_t v_{rr}\right|\le 2\int r^{-n+4} |v_t| |v_r| \leq  8\int  r^{-n+4} v_t^2 + \frac12\int  r^{-n+4} v_{rr}^2.
	  \end{align}
	
		Therefore, we conclude that 
		\begin{align}
		\label{est:IV}
			 IV &\geq -C\alpha^4\ve\int r^{-n+\ve}v^2 +\frac{39}{10}\alpha^2\ve \int r^{-n+\ve}|\nabla_{\sn}v|^2-\alpha^2\int r^{-n+2}v_{r}^2\\&-33\int r^{-n+4}v_{t}^2-\frac12\int r^{-n+4}v_{rr}^2\notag
		\end{align}
		Hence from \eqref{rt12nd}-\eqref{est:IV} , it follow
		\begin{align}
		\label{est:234}
		II+III+IV&\geq \alpha^2\int r^{-n+2}v_{r}^2+\frac12 \int r^{-n+4}v_{rr}^2 
		-C\alpha^4\varepsilon\int r^{-n+\varepsilon}v^2-33\int r^{-n+4}v_{t}^2\\
		&+ 2 \int r^{-n+4}|\nabla_Tv_r|^2 +\frac{39}{10}\alpha^2\ve \int r^{-n+\ve}|\nabla_{\sn}v|^2.\notag 
		\end{align}

		Now we are just left with estimating the integral $I.$ 
		
		\medskip

		\noindent\textbf{{\it\underline{Estimates for $I$:}}} We estimate I in two ways in order  to incorporate a  critical zero order term and also a   tangential second derivative term. First we estimate I using the spherical harmonic decomposition of $L^2(\sn).$ This part of the argument is similar to that in \cite{BGM} which uses the spectral gap of the spherical Laplacian. However  in our estimate, we also additionally incorporate a tangential gradient term that  is required eventually.  Thus we  provide all the details. For the sake of the convenience of the reader, we recall some notations and basics facts about spherical harmonics here. 
		
		Let $H_k$ denote the space of homogeneous harmonic polynomials of degree $k\geq 0.$ Restrictions of elements of $H_k$ to $\sn$ are called the spherical harmonics of degree $k$, which is denoted by $\mathcal{H}_k.$ The spherical harmonic decomposition reads as $L^2(\sn)=\bigoplus_{k\geq 0}\mathcal{H}_k.$ Here $\mathcal{H}_k$'s are finite dimensional subspaces, an orthonormal basis of which is given by $\{S_{k,j}:1\leq j\leq d_k\}$, where $d_k$ denotes the dimension of $\mathcal{H}_k.$ Then it is well known that  $\{S_{k,j}: 1\leq j\leq d_k, k\geq0\}$ forms an orthonormal basis for $L^2(\sn)$, and  $\Delta_{\sn} S_{k,j} = - k(k+n-2) S_{k,j}$. So, writing $v(x,t) = v(r\omega,t)$, the expansion of $v$ in terms of spherical harmonics takes the form 
		$$v(rw,t)=\sum_{k=0}^\infty\sum_{j=1}^{d_k} v_{k,j}(r,t) S_{k,j}(w)$$ 
		where $v_{k,j}$ dnote the spherical harmonic coefficient of $v$ defined by $$v_k(r,t) = \int_{\sn} v(rw,t) S_{k,j}(w) d\sigma(w).$$
		This leads to the following spectral decomposition of the spherical Laplacian:
		$$\Delta_{\sn} v(rw,t) = - \sum_{k=0}^\infty \sum_{j=1}^{d_k}k(k+n-2) v_k(r,t) S_{k,j}(w)$$
		which will help us in estimating $I.$ This decomposition along with the orthonormality of spherical harmonics gives  
		\begin{align*}
			& \int r^{-2\alpha+2\beta-4} (\beta(\beta+n-2) v + \Delta_{\sn}v )^2 \\
			&= \int r^{-n}( \beta(\beta+n-2) v + \Delta_{\sn} v)^2
			\\
			& =  \int_{0}^T \int_{0}^{\infty} r^{-1} \sum_{k=0}^{\infty}\sum_{j=1}^{d_k} ( \beta(\beta+n-2) - k(k+n-2))^2 v_{k,j} (r,t)^2 dr dt.
		\end{align*}
		Now given that $\alpha$ is of the form $\alpha= k + \frac{n+1}{2}$, in view of \eqref{rel} it follows that $\text{dist} (\beta, \mathbb N) = \frac{1}{2}.$ This implies the following inequality
		$$( \beta(\beta+n-2) - k(k+n-2))^2\geq \frac12 \beta(\beta+n-2) +\frac12k(k+n-2)$$ which yields
		\begin{align*}
		I&\geq \int_{0}^T \int_{0}^{\infty} r^{-1} \sum_{k=0}^{\infty}\sum_{j=1}^{d_k} ( \frac12\beta(\beta+n-2)+\frac12
		k(k+n-2)) v_{k,j} (r,t)^2 dr dt\\
		&= \frac12\beta(\beta+n-2)\int r^{-n}v^2+\frac12 \int r^{-n}|\nabla_{\sn}v|^2.
		\end{align*}
		Now in in view of \eqref{rel}, we have that  $\frac12\beta(\beta+n-2)\geq \frac{\alpha^2}{2}$ which leads to the estimate 
		\begin{equation}
		\label{est:I1}
		I\geq \frac{\alpha^2}{2}\int r^{-n}v^2+\frac12 \int r^{-n+2}|\nabla_{T}v|^2.
		\end{equation}
		
		Before we estimate  $I$ in another way, recalling \eqref{hj} we  note that 
$$(\nabla_Tv)_r=\partial_r(r^{-1}\nabla_{\sn}v)=-r^{-2}\nabla_{\sn}v+r^{-1}\nabla_{\sn}v_r=-r^{-1}\nabla_Tv+\nabla_Tv_r$$ leading to
\begin{align*}
	&\int r^{-n+4}|\nabla_Tv_r|^2=\int r^{-n+4}|(\nabla_Tv)_r+r^{-1}\nabla_Tv|^2\\
	&= \int r^{-n+4}|(\nabla_Tv)_r|^2+2\int r^{-n+3}\langle (\nabla_Tv)_r, \nabla_Tv \rangle +\int r^{-n+2}|\nabla_Tv|^2.
\end{align*}
Now using the algebraic inequality in \eqref{ai1} ( with $\delta=\frac{1}{2}$) we find that for some $0 <c_0 <1/4$ 
\begin{align}\label{jku1}
& \int r^{-n+4}|(\nabla_Tv)_r|^2+2\int r^{-n+3}\langle (\nabla_Tv)_r, \nabla_Tv \rangle +\int r^{-n+2}|\nabla_Tv|^2 +  \frac12 \int r^{-n+2}|\nabla_{T}v|^2\\
& \geq       c_0 \left(  \int r^{-n+2}|\nabla_{T}v|^2 + \int r^{-n+4}|(\nabla_Tv)_r|^2\right).\notag
\end{align}
Thus from \eqref{est:234} it follows  by writing \begin{equation}\label{brk}2 \int r^{-n+4} |\nabla_T v_r|^2 =  \int r^{-n+4} |\nabla_T v_r|^2 +\int r^{-n+4} |\nabla_T v_r|^2 \end{equation} and by applying the estimate in \eqref{jku1}  to one of the term in the right hand side of \eqref{brk} above that the following inequality holds
\begin{align}\label{one}
   &  \int r^{-2\alpha}  e^{2\alpha r^\ve}(\Delta u - \partial_t u)^2 \geq\alpha^2\int r^{-n+2}v_{r}^2+\frac12 \int r^{-n+4}v_{rr}^2 \\
		&-C\alpha^4\varepsilon\int r^{-n+\varepsilon}v^2-33\int r^{-n+4}v_{t}^2\notag\\
		&+  \int r^{-n+4}|\nabla_Tv_r|^2  +  \frac{\alpha^2}{2}\int r^{-n}v^2 +   c_0 \left(  \int r^{-n+2}|\nabla_{T}v|^2 + \int r^{-n+4}|(\nabla_Tv)_r|^2\right).\notag 	
\end{align}
		
		Now we estimate $I$ by simply expanding the square in the integrand as follows:
		\begin{align}
		\label{est:I2}
		I&=\int r^{-2\alpha+2\beta-4} (\beta(\beta+n-2) v + \Delta_{\sn}v )^2\\
		&=\int r^{-n}(\Delta_{\sn}v)^2+(\beta(\beta+n-2))^2\int r^{-n}v^2+2\beta(\beta+n-2)\int r^{-n}v\Delta_{\sn}v\notag \\
		&\geq \int r^{-n}(\Delta_{\sn}v)^2+\alpha^4\int r^{-n}v^2-2\beta(\beta+n-2)\int r^{-n}|\nabla_{\sn}v|^2.\notag
		\end{align}

		Likewise  from  \eqref{est:234} and \eqref{est:I2} we deduce the following estimate
		\begin{align}
		\label{sf-2}
		\int r^{-2\alpha}  e^{2\alpha r^\ve}&(\Delta u - \partial_t u)^2\geq \int r^{-n}(\Delta_{\sn}v)^2+\alpha^4\int r^{-n}v^2-2\beta(\beta+n-2)\int r^{-n}|\nabla_{\sn}v|^2\\  &+\alpha^2\int r^{-n+2}v_{r}^2+\frac12 \int r^{-n+4}v_{rr}^2 + 2\int r^{-n+4} |\nabla_T v_r|^2 \notag\\& -C\alpha^4\varepsilon\int r^{-n+\varepsilon}v^2-33\int r^{-n+4}v_{t}^2.\notag
		\end{align}
	We now employ an idea from \cite{Reg}.	By multiplying \eqref{one} with $\frac{8\beta}{c_0}(\beta+n-1)$ and then adding with \eqref{sf-2}, we obtain 
		\begin{align*}
		&\left(1+\frac{8\beta}{c_0}(\beta+n-1)\right)\int r^{-2\alpha}  e^{2\alpha r^\ve}(\Delta u - \partial_t u)^2\\
		&\geq \frac{4\beta}{c_0}(\beta+n-1)\alpha^2\int r^{-n}v^2 + 4 \beta^2\left(  \int r^{-n+2}|\nabla_{T}v|^2 + \int r^{-n+4}|(\nabla_Tv)_r|^2\right) 	+\alpha^4\int r^{-n+2}v_{r}^2\\	
		&+\frac{\alpha^2}{ 2} \int r^{-n+4}v_{rr}^2+\int r^{-n}(\Delta_{\sn}v)^2
		-(1+\frac{8\beta}{c_0}(\beta+n-1))[C\alpha^4\varepsilon\int r^{-n+\varepsilon}v^2+33\int r^{-n+4}v_{t}^2],
		\end{align*}
		which on dividing both side by $(1+\frac{8\beta}{c_0}(\beta+n-1))$ yields
		\begin{align}
		&\int r^{-2\alpha}  e^{2\alpha r^\ve}(\Delta u - \partial_t u)^2\geq C_1\alpha^2 \int r^{-n}v^2+ C_2 \left[\int r^{-n}|\nabla_{\sn}v|^2+\int r^{-n+2}v_{r}^2\right]\\
		&+\frac{C_3}{\alpha^2}\left[\int r^{-n+4}v_{rr}^2+\int r^{-n}(\Delta_{\sn}v)^2+ \int r^{-n+4}|(\nabla_Tv)_r|^2+\int r^{-n+4}|\nabla_Tv_r|^2\right]\notag\\&-C\alpha^4\varepsilon\int r^{-n+\varepsilon}v^2-33\int r^{-n+4}v_{t}^2.\notag
		\end{align}
	Such an inequality   is valid for sufficiently large $\alpha$ of the form $k + \frac{n+1}{2}$  and  sufficiently small $\varepsilon.$ Now we put $v = r^{-\beta} e^{\alpha r^\ve} u$ in the above inequality and estimate each term in the RHS in terms of $u.$ One immediately has 
		\begin{align}
        \label{sf-3}
		&\int r^{-2\alpha}  e^{2\alpha r^\ve}(\Delta u - \partial_t u)^2\geq C_1\alpha^2 \int r^{-2\alpha-4}e^{\alpha r^{\ve}}u^2+ C_2 \left[\int r^{-2\alpha-4}e^{2\alpha r^{\ve}}|\nabla_{\sn}u|^2+ \int r^{-n+2}v_{r}^2\right]\\
		&+\frac{C_3}{\alpha^2}\int r^{-n+4}v_{rr}^2+\frac{C_3}{\alpha^2}\int r^{-2\alpha-4}e^{2\alpha r^{\ve}}(\Delta_{\sn}u)^2+\frac{C_3}{\alpha^2} \int r^{-n+4}|(\nabla_Tv)_r|^2+\frac{C_3}{\alpha^2}\int r^{-n+4}|\nabla_Tv_r|^2\notag\\&-C\alpha^4\varepsilon\int r^{-2\alpha-4+\varepsilon}e^{2\alpha r^{\ve}}u^2-33\int r^{-2\alpha}e^{2\alpha r^{\ve}}u_{t}^2.\notag
		\end{align} We first handle the integrals involving the terms $v_r.$ Note that by an easy calculation we have 
		$$v_r=e^{\alpha r^{\ve}}(r^{-\beta}u_r-\beta r^{-\beta-1}u+r^{-\beta+\ve-1}\alpha\ve u).$$
		Substituting this in \eqref{sf-3}  and by using \eqref{rel} we get
		\begin{align}\label{jkh1}
		&\int r^{-2\alpha}  e^{2\alpha r^\ve}(\Delta u - \partial_t u)^2\geq C_1\alpha^2 \int r^{-2\alpha-4}e^{\alpha r^{\ve}}u^2+ C_2 \bigg[\int r^{-2\alpha-4}e^{2\alpha r^{\ve}}|\nabla_{\sn}u|^2+ \int r^{-2\alpha -2}e^{2\alpha r^\ve}u_{r}^2\\
		& + \int r^{-2\alpha -4} e^{2 \alpha r^\ve} (\beta u-r^{\ve}\alpha\ve u)^2 - 2 \int r^{-2\alpha -3} e^{2\alpha r^\ve} u_r (\beta u - r^\ve \alpha \ve u ) \bigg]\notag\\ &  +\frac{C_3}{\alpha^2}\int r^{-n+4}v_{rr}^2+\frac{C_3}{\alpha^2}\int r^{-2\alpha-4}e^{2\alpha r^{\ve}}(\Delta_{\sn}u)^2+\frac{C_3}{\alpha^2} \int r^{-n+4}|(\nabla_Tv)_r|^2+\frac{C_3}{\alpha^2}\int r^{-n+4}|\nabla_Tv_r|^2\notag\\&-C\alpha^4\varepsilon\int r^{-2\alpha-4+\varepsilon}e^{2\alpha r^{\ve}}u^2-33\int r^{-2\alpha}e^{2\alpha r^{\ve}}u_{t}^2.\notag
		\end{align} 
		Now by using the inequality \eqref{ai1} we find, 
		\begin{align}\label{aiop}
		&C_1\alpha^2 \int r^{-2\alpha-4}e^{2\alpha r^{\ve}}u^2 	 + C_2   \bigg[ \int r^{-2\alpha -2}e^{2\alpha r^\ve}u_{r}^2\\
		& + \int r^{-2\alpha -4} e^{2 \alpha r^\ve} (\beta u-r^{\ve}\alpha\ve u)^2 - 2 \int r^{-2\alpha -3} e^{2\alpha r^\ve} u_r (\beta u - r^\ve \alpha \ve u ) \bigg]	\notag\\
		& \geq  c_1   \left( \alpha^2 \int r^{-2\alpha-4}e^{2\alpha r^{\ve}}u^2 	  +\int r^{-2\alpha -2} e^{2\alpha r^\ve} u_r^2\right).\notag
			\end{align}	
			Using \eqref{aiop} in \eqref{jkh1} we have
			\begin{align}\label{jkh2}
		&\int r^{-2\alpha}  e^{2\alpha r^\ve}(\Delta u - \partial_t u)^2\geq c_1\alpha^2 \int r^{-2\alpha-4}e^{\alpha r^{\ve}}u^2+ c_1 \bigg[\int r^{-2\alpha-4}e^{2\alpha r^{\ve}}|\nabla_{\sn}u|^2+ \int r^{-2\alpha -2}e^{2\alpha r^\ve}u_{r}^2 \bigg]\\
		 &  +\frac{C_3}{\alpha^2}\int r^{-n+4}v_{rr}^2+\frac{C_3}{\alpha^2}\int r^{-2\alpha-4}e^{2\alpha r^{\ve}}(\Delta_{\sn}u)^2+\frac{C_3}{\alpha^2} \int r^{-n+4}|(\nabla_Tv)_r|^2+\frac{C_3}{\alpha^2}\int r^{-n+4}|\nabla_Tv_r|^2\notag\\&-C\alpha^4\varepsilon\int r^{-2\alpha-4+\varepsilon}e^{2\alpha r^{\ve}}u^2-33\int r^{-2\alpha}e^{2\alpha r^{\ve}}u_{t}^2.\notag
		\end{align} 			
			
		Now we handle the integral involving $(\nabla_Tv)_r.$ An easy calculation yields:
		$$(\nabla_Tv)_r=e^{\alpha r^{\ve}}r^{-\beta}(\nabla_Tu)_r+e^{\alpha r^{\ve}}(\alpha\ve r^{-\beta+\ve-1}-\beta r^{-\beta-1})\nabla_Tu.$$ Using this along with  the algebraic inequality  in \eqref{ai1}  we can proceed as above and deduce the following inequality
		\begin{align}
			\label{dtvr}
			&\frac{C_3}{\alpha^2} \int r^{-n+4}|(\nabla_Tv)_r|^2 + c_1 \int r^{-2\alpha -4} e^{2\alpha r^\ve} |\nabla_{\sn} u|^2 \\&\geq \frac{c_2}{\alpha^2} \int r^{-2\alpha}e^{2\alpha r^{\ve}}|(\nabla_Tu)_r|^2 + c_2 \int r^{-2\alpha -2} e^{2\alpha r^{\ve}} |\nabla_T u|^2.\notag
		\end{align}
		Using \eqref{dtvr} in \eqref{jkh2} we find
		\begin{align}\label{jkh4}
		&\int r^{-2\alpha}  e^{2\alpha r^\ve}(\Delta u - \partial_t u)^2\geq c_1\alpha^2 \int r^{-2\alpha-4}e^{\alpha r^{\ve}}u^2+ c_2 \bigg[\int r^{-2\alpha-2}e^{2\alpha r^{\ve}}|\nabla_{T}u|^2+ \int r^{-2\alpha -2}e^{2\alpha r^\ve}u_{r}^2 \bigg]\\
		 &  +\frac{C_3}{\alpha^2}\int r^{-n+4}v_{rr}^2+\frac{C_3}{\alpha^2}\int r^{-2\alpha-4}e^{2\alpha r^{\ve}}(\Delta_{\sn}u)^2+\frac{c_2}{\alpha^2} \int r^{-2\alpha}e^{2\alpha r^{\ve}}|(\nabla_Tu)_r|^2+\frac{C_3}{\alpha^2}\int r^{-n+4}|\nabla_Tv_r|^2\notag\\&-C\alpha^4\varepsilon\int r^{-2\alpha-4+\varepsilon}e^{2\alpha r^{\ve}}u^2-33\int r^{-2\alpha}e^{2\alpha r^{\ve}}u_{t}^2.\notag
		\end{align} 	
		Similarly using 
		$$(\nabla_Tu)_r =- r^{-1}\nabla_Tu + \nabla_Tu_r$$ combined with the  algebraic inequality \eqref{ai1} we observe
		\begin{equation}\label{jkopp}
		\frac{c_2}{2\alpha^2} \int r^{-2\alpha}e^{2\alpha r^{\ve}}|(\nabla_Tu)_r|^2+ 	c_2\int r^{-2\alpha-2}e^{2\alpha r^{\ve}}|\nabla_{T}u|^2	 \geq c_4  \int r^{-2\alpha-2}e^{2\alpha r^{\ve}}|\nabla_{T}u|^2 + \frac{c_4}{\alpha^2} \int r^{-2\alpha} |\nabla_T u_r|^2.		\end{equation}
		Therefore one can incorporate  $\frac{c_4}{\alpha^2} \int r^{-2\alpha} e^{2\alpha r^\ve} |\nabla_T u_r|^2$ on the right hand side in \eqref{jkh4} above. Finally by writing		 $v_{rr}$ in terms of $u, u_r$ and $u_{rr}$	 and again by using \eqref{ai1} we  find
		\begin{align}\label{jkopp1}
		&c_1\alpha^2 \int r^{-2\alpha-4}e^{2\alpha r^{\ve}}u^2	 + c_2 \int r^{-2\alpha -2} e^{2\alpha r^\ve} u_r^2 + 	\frac{C_3}{\alpha^2}\int r^{-n+4}v_{rr}^2	
		\\
		& \geq c_5 \alpha^2 \int r^{-2\alpha -4} e^{2\alpha r^\ve} u^2 + c_5 \int r^{-2\alpha -2} e^{2\alpha r^\ve} u_r^2 + \frac{c_5}{\alpha^2} \int r^{-2\alpha} e^{2 \alpha r^\ve} u_{rr}^2.\notag\end{align}
		Indeed, by an easy calculation we first note that
		\begin{align*}
		   e^{-\alpha r^{\ve}} v_{rr}&=r^{-\beta-2}[\beta(\beta+1)+\alpha\ve(-\beta+\ve-1)r^{\ve}-\alpha\ve\beta r^{\ve}+\alpha^2\ve^2r^{2\ve}]u 
		   +r^{-\beta-1}(2\alpha\ve r^{\ve}-2\beta)u_r+ r^{-\beta}u_{rr}.
		\end{align*}
		Let us write $v_{rr}/\alpha=a+b$ with $a=e^{\alpha r^{\ve}}r^{-\beta}u_{rr}/\alpha$ and $b=e^{\alpha r^{\ve}}r^{-\beta-2}(F_1u+F_2u_r)$ where $F_1$ and $F_2$ are given by 
		\[F_1=\beta(\beta+1)/\alpha+\ve(-\beta+\ve-1)r^{\ve}-\ve\beta r^{\ve}+\alpha\ve^2r^{2\ve},~\text{and}~F_2:=(2\ve r^{\ve}-2\beta/\alpha)r.\]
		Now applying the algebraic identity $(a+b)^2=a^2+2ab+b^2$ we obtain 
		\begin{align*}
		    &\frac{1}{\alpha^2}\int r^{-n+4}v_{rr}^2\geq \frac{1}{\alpha^2} \int r^{-2\alpha} e^{2 \alpha r^\ve} u_{rr}^2+\int r^{-2\alpha-4} e^{2 \alpha r^\ve}(F_1u+F_2u_r)^2+\frac{2}{\alpha}\int r^{-2\alpha-2} e^{2 \alpha r^\ve}(F_1u+F_2u_r)u_{rr}
		\end{align*}
		Now observe that 
		\begin{align*}
		    \int r^{-2\alpha-4} e^{2 \alpha r^\ve}(F_1u+F_2u_r)^2&\leq 2\int r^{-2\alpha-4} e^{2 \alpha r^\ve}(F_1^2u^2+F_2^2u_r^2)\\&=2\int r^{-2\alpha-4} e^{2 \alpha r^\ve}F_1^2u^2+2\int r^{-2\alpha-2} e^{2 \alpha r^\ve}(2\ve r^{\ve}-2\beta/\alpha)^2u_r^2\\
		    &\leq 2a_1\alpha^2 \int r^{-2\alpha-4} e^{2 \alpha r^\ve}u^2+2a_2\int r^{-2\alpha-2} e^{2 \alpha r^\ve}u_r^2
		\end{align*}
		where in the last inequality we have used the facts that $F_1^2\leq a_1 \alpha^2$ and $(2\ve r^{\ve}-2\beta/\alpha)^2\leq a_2$ for some positive constants $a_1$ and $a_2.$
		Hence we can choose $\delta>0$ (small enough) such that 
		\begin{align*}
		    \frac{c_1\alpha^2}2 \int r^{-2\alpha-4}e^{2\alpha r^{\ve}}u^2	 + \frac{c_2}2 \int r^{-2\alpha -2} e^{2\alpha r^\ve} u_r^2 \geq \delta C_3  \int r^{-2\alpha-4} e^{2 \alpha r^\ve}(F_1u+F_2u_r)^2
		\end{align*}
		which, together with the above observations, yields 
		\begin{align*}
		   & c_1\alpha^2 \int r^{-2\alpha-4}e^{2\alpha r^{\ve}}u^2	 + c_2 \int r^{-2\alpha -2} e^{2\alpha r^\ve} u_r^2 + 	\frac{C_3}{\alpha^2}\int r^{-n+4}v_{rr}^2 \\&\geq  \frac{c_1\alpha^2}2 \int r^{-2\alpha-4}e^{2\alpha r^{\ve}}u^2	 + \frac{c_2}2 \int r^{-2\alpha -2} e^{2\alpha r^\ve} u_r^2\\
		   &+ C_3 \left[(1+\delta)\int r^{-2\alpha-4} e^{2 \alpha r^\ve}(F_1u+F_2u_r)^2+\frac{2}{\alpha}\int r^{-2\alpha-2} e^{2 \alpha r^\ve}(F_1u+F_2u_r)u_{rr}+ \frac{1}{\alpha^2}\int r^{-2\alpha} e^{2 \alpha r^\ve} u_{rr}^2\right]
		\end{align*}
		 from which \eqref{jkopp1} follows by an application of Lemma \ref{alineq}.

		Using \eqref{jkopp} and \eqref{jkopp1} in \eqref{jkh4} and also the fact that $\int_{\sn} (\Delta_{\sn} f)^2 \geq C_n \int_{\sn} (\nabla^2_{\sn} f)^2$ which is a consequence of the ellipticity of the spherical Laplacian, we finally deduce the following estimate for some new constants $C_1, C_2$ and  $C_3$
		 		 
		 	\begin{align}
		 \label{sf-5}
		 &\int r^{-2\alpha}  e^{2\alpha r^\ve}(\Delta u - \partial_t u)^2\geq C_1\alpha^2 \int r^{-2\alpha-4}e^{\alpha r^{\ve}}u^2+ C_2 \int r^{-2\alpha-2}e^{2\alpha r^{\ve}}|\nabla u|^2\\
		 &+\frac{C_3}{\alpha^2}\int r^{-2\alpha}e^{2\alpha r^{\ve}}|\nabla^2u|^2-C\alpha^4\varepsilon\int r^{-2\alpha-4+\varepsilon}e^{2\alpha r^{\ve}}u^2-33\int r^{-2\alpha}e^{2\alpha r^{\ve}}u_{t}^2.\notag
		 \end{align}
		 Now to get rid of the integral involving the $u_t^2$- term, we use the following lemma proved in \cite{BGM}. 
		\begin{lemma}\label{L:del} 
			Let $R<1$ and let  $u \in C_0^{\infty}((B_R \setminus \{0\}) \times (0,T))$. There exist constants $d= d(n)>0$, $\alpha(n)>>1$ and $0<\ve(n)<<1$, such that for all $\alpha\ge \alpha(n)$ and every $0<\ve<\ve(n)$ one has 
			\begin{align}\label{claim1}
				\frac{d}{\alpha} \int r^{-2\alpha} e^{2\alpha r^\ve} u_t^2 + d\alpha^3\ve^2 \int r^{-2\alpha-4+\ve} e^{2\alpha r^\ve} u^2   \leq \int r^{-2\alpha}  e^{2\alpha r^\ve} (\Delta u - u_t)^2.
			\end{align}
		\end{lemma}
	
		Having the above lemma in hand, let us fix $0<\ve(n)<1$ and $\alpha(n)>>1$ such that \eqref{sf-5} and \eqref{claim1} hold simultaneously for $0<\ve<\ve(n)$ and $\alpha>\alpha(n)$. Now we pick $d_0 = d_0(n,\ve)>1$ suitably in such that
			$d_0 d  \ve  \geq 2 C \ \ \ \text{and}\ \ \ d d_0 > 33.$
		Having chosen such $d_0$, we multiply \eqref{claim1} by $d_0\alpha$ and add the resulting inequality to \eqref{sf-5}, obtaining
		\begin{align}\label{hold2}
		 &(d_0 \alpha +1 )\int r^{-2\alpha} e^{2\alpha r^\ve} (\Delta u - u_t)^2\geq C_1\alpha^2 \int r^{-2\alpha-4}e^{\alpha r^{\ve}}u^2+ C_2 \int r^{-2\alpha-2}e^{2\alpha r^{\ve}}|\nabla u|^2\\
		 &+\frac{C_3}{\alpha^2}\int r^{-2\alpha}e^{2\alpha r^{\ve}}|\nabla^2u|^2+(d_0d\ve-C)\alpha^4 \ve\int r^{-2\alpha-4+\varepsilon}e^{2\alpha r^{\ve}}u^2+(dd_0-33)\int r^{-2\alpha}e^{2\alpha r^{\ve}}u_{t}^2.\notag
			\notag
		\end{align}
	Finally, dividing both sides of the above inequality by  $\alpha$,  the required Carleman estimate \eqref{est1} is seen to follow.
		
	\end{proof}	 
	 
\subsection{Exponential decay of solutions}
Using the Carleman estimate in \eqref{est1}, we now show that if a solutions $u$ to the differential inequality \eqref{main-diff-ineq} vanishes to infinite order in the sense of  \eqref{vp},  then it exponential decays in space.
We first recall following Caccioppoli type estimate,  the proof of which is exactly the same as that of \cite[Lemma 3.1]{BGM}.
\begin{lemma}\label{cac-bgg}
Let $u$ be a solution to \eqref{main-diff-ineq}  in $B_R \times (-T, T)$ and let $0 < a< 1 < b$. Then, there exists a constant $C_1>0$, depending on $n, a, b, T$ and $M$ in \eqref{main-diff-ineq}, such that for every $r< \min\{1, R\}$ the following holds 
\[
\int_{ \{r/2<|x| < r\} \times (-T/2, T/2)} |\nabla u|^2 \leq \frac{C_1}{r^2} \int_{ \{r(1-a)/2 < |x|<  b r\} \times (-T, T) } u^2.
\]
\end{lemma}

\subsection{Exponential decay}

In order to proceed further, we fix some notations. Let $g=(g_{ij}(x, t))$ denotes the inverse of the coefficient matrix $A(x, t)$.  Consider the following weight.
\begin{align}
\label{weight-log-estimate}
\sigma(x,t)=\left(\sum_{i, j=1}^{n}g_{ij}(0, t)x_{i} x_{j}\right)^{\frac{1}{2}}.
\end{align}
In view of the  uniform ellipticity of $A$, it is not hard to see that \begin{equation}
\label{sigmaest}
    M|x|\leq \sigma(x,t)\leq N|x|
\end{equation}
where $N, M$ are  constants depending on the ellipticity constant of $A.$ We set \begin{equation}\label{lam}\lambda:=N/M\geq 1.\end{equation} With this new weight, we have the following Carleman estimate for the operator under consideration which is derived from \eqref{est1} using an appropriate change of variable.
\begin{lemma}
\label{carlemanL}
Let $A$ be as  in \eqref{assump}. For sufficiently large $\alpha$ of the form $ \alpha=\frac{n+1}{2} + k$  where $k \in \mathbb N$ and small $0<\ve=\ve(n)<<1$,  we have that  for $R_0 \leq c_0 \alpha^{-3/2}$ with $c_0$ sufficiently small, the following estimate  holds  for $v \in C_0^{\infty}((B_{R_0} \setminus \{0\}) \times (0,T))$. \begin{align}
	\label{L-1}
		&\alpha \int |\sigma(x,t)|^{-2\alpha-4} e^{2\alpha \sigma(x,t)^{\ve}} \, v^2  dxdt  +\alpha^3\int \sigma(x,t)^{-2\alpha- 4+\ve} e^{2\alpha \sigma(x,t)^{\ve}}  v^2dxdt \\  
	&+\frac{1}{\alpha}\int \sigma(x,t)^{-2\alpha-2} e^{2\alpha \sigma(x,t)^{\ve}} \, |\nabla v|^2  dxdt+\frac{1}{\alpha^3}\int \sigma(x,t)^{-2\alpha} e^{2\alpha \sigma(x,t)^{\ve}} \, |\nabla^2v|^2  dxdt\nonumber\\
	&\leq C\int \sigma(x,t)^{-2\alpha} e^{2\alpha \sigma(x,t)^{\ve}} (\operatorname{div}(A(x,t)\nabla v)-\partial_t v)^2  dxdt \nonumber
\end{align}

\end{lemma}
\begin{proof}
To prove the Carleman type estimate \eqref{L-1}, we perform a suitable change of variable to the previous Carleman estimate \eqref{est1} for the heat operator. To start with, let $P(t):=(p_{ij}(t))_{n\times n}$ stand for the positive square root of the matrix $A(0,t).$  Now we  apply  the following change of variable $$y=P(s)x,\\~s=t$$ on the both side of our previous Carleman estimate \eqref{est1} for the heat operator.  Now under this change of variable, the transformations of the involved differential operators can be obtained by a straightforward calculation. However, for the sake of the convenience of the reader, we provide all the details. To begin with, we write 
   	\[u(x,t)=u(P(s)^{-1}y,s)=:v(y,s),~~y=P(s)x.\]
   	We now observe that \[\partial_iu=\sum_{k=1}^np_{ik}\partial_kv,~\text{and}~\partial_i\partial_ju=\sum_{k=1}^n\sum_{l=1}^{n}p_{ik}p_{lj}\partial_k\partial_lv\] leading to 
   	\[\nabla u=P(s).\nabla v, ~\Delta u= Tr(A(0,s).\nabla^2v),~\text{and}~ \nabla^2u=P(s)\nabla^2vP(s).\]
   	Also we note that
   	\[\partial_tu=((P)_s.x).\nabla v+\partial_sv=C(s,y).\nabla v+\partial_sv \]
   	where $C(s,y):=(P)_s(P(s))^{-1}y$ and $(P)_s$ is the $s$-partial derivative of the matrix $P(s)$.  Furthermore
   	$$|x|=|P(s)^{-1}y|=\left(\sum_{i, j=1}^ng_{ij}(0,t)y_iy_j\right)^{\frac12}=\sigma(y,s).$$ Note that $(g_{ij}(0,t))$ is the inverse of $A(0,t)$.
The above observations, in view of the aforementioned change of variable transforms \eqref{est1} to 
\begin{align}
	\label{L--1}
		&\alpha \int \sigma(y,s)^{-2\alpha-4} e^{2\alpha \sigma(y,s)^{\ve}} \, v^2  dyds  +\alpha^3\int \sigma(y,s)^{-2\alpha- 4+\ve} e^{2\alpha \sigma(y,s)^{\ve}}  v^2dyds  \\  
	&+\frac{1}{\alpha}\int \sigma(y,s)^{-2\alpha-2} e^{2\alpha \sigma(y,s)^{\ve}} \, |P(s).\nabla v|^2  dyds+\frac{1}{\alpha^3}\int \sigma(y,s)^{-2\alpha} e^{2\alpha \sigma(y,s)^{\ve}} \, |P(s)\nabla^2vP(s)|^2  dyds\nonumber\\
	&\leq C\int \sigma(y,s)^{-2\alpha} e^{2\alpha \sigma(y,s)^{\ve}} (Tr(A(0,s).\nabla^2v)-C(s,y).\nabla v-\partial_sv)^2   dy ds \nonumber
\end{align}
which clearly is valid for all $v\in C^{\infty}_0(B_{R_0}\times(0,T))$ where $R_0$ is small and depends also on the ellipticity of $A.$
	Now note that 
\begin{align*}
	(Tr(A(0,s).\nabla^2v)-C(s,y).\nabla v-\partial_sv)^2\leq 2|\operatorname{div}(A(y,s)\nabla v)-\partial_s v|^2+ 2I
\end{align*}
where the term $I$ is defined and estimated as follows using \eqref{assump}:
 \begin{align*}
	I:&=|Tr(A(0,s).\nabla^2v)-C(s,y).\nabla v-\operatorname{div}(A(y,s)\nabla v)|^2\\
	&\leq C_1( |\nabla v|^2+ |y|^2|\nabla^2v|^2).
\end{align*}
For convenience, changing the variable from $y$ to $x$ and $s$ to $t$,  in view of this observation and $\sigma(x,t)\leq N |x|$, from \eqref{L--1}, we obtain 
\begin{align}
	&\int \sigma(x,t)^{-2\alpha}e^{2\alpha\sigma(x,t)^\ve} (\operatorname{div}(A(x,t)\nabla v)-\partial_t v)^2dxdt \geq \\
		&+\alpha \int \sigma(x,t)^{-2\alpha-4} e^{2\alpha \sigma(x,t)^{\ve}} \, v^2  dxdt  +\alpha^3\int \sigma(x,t)^{-2\alpha- 4+\ve} e^{2\alpha \sigma(x,t)^{\ve}}  v^2dxdt \notag \\  
	&+\frac{1}{\alpha}\int\sigma(x,t)^{-2\alpha-2} e^{2\alpha \sigma(x,t)^{\ve}} \, |\nabla v|^2  dxdt+\frac{1}{\alpha^3}\int \sigma(x,t)^{-2\alpha} e^{2\alpha \sigma(x,t)^{\ve}} \, |\nabla^2v|^2  dxdt\nonumber\\
	&	-C_1N^2\int \sigma(x,t)^{-2\alpha-2}|x|^2e^{2\alpha\sigma(x,t)^\ve}|\nabla v|^2dxdt-	C_1\int \sigma(x,t)^{-2\alpha}|x|^2e^{2\alpha\sigma(x,t)^\ve}|\nabla^2v|^2dxdt\nonumber.
\end{align}
We  now take $R_0\leq c_0\alpha^{-3/2}$ for some suitable constant $c_0$ to be chosen later, then the above inequality transforms to 
\begin{align}
    &\int \sigma(x,t)^{-2\alpha}e^{2\alpha\sigma(x,t)^\ve} (\operatorname{div}(A(x,t)\nabla v)-\partial_t v)^2dxdt \geq \\
		&+\alpha \int \sigma(x,t)^{-2\alpha-4} e^{2\alpha \sigma(x,t)^{\ve}} \, v^2  dxdt  +\alpha^3\int \sigma(x,t)^{-2\alpha- 4+\ve} e^{2\alpha \sigma(x,t)^{\ve}}  v^2dxdt \notag \\  
	&+\frac{1}{\alpha}(1-C_1N^2c_0^2\alpha^{-2})\int\sigma(x,t)^{-2\alpha-2} e^{2\alpha \sigma(x,t)^{\ve}} \, |\nabla v|^2  dxdt\notag \\&+\frac{1}{\alpha^3}(1-C_2c_0^2)\int \sigma(x,t)^{-2\alpha} e^{2\alpha \sigma(x,t)^{\ve}} \, |\nabla^2v|^2  dxdt.\nonumber
\end{align}
Choosing $c_0$ such that $C_2c_0^2<1/2$, we  thus conclude that 
\begin{align}
	\label{est:f1}
		&C\int \sigma(x,t)^{-2\alpha}e^{2\alpha|x|^\ve} (\operatorname{div}(A(x,t)\nabla v)-\partial_t v)^2dxdt \geq \\
	&+\alpha \int\sigma(x,t)^{-2\alpha-4} e^{2\alpha \sigma(x,t)^{\ve}} \, v^2  dxdt  +\alpha^3\int \sigma(x,t)^{-2\alpha- 4+\ve} e^{2\alpha \sigma(x,t)^{\ve}}  v^2dxdt \notag \\  
	&+\frac{1}{2\alpha}\int \sigma(x,t)^{-2\alpha-2} e^{2\alpha\sigma(x,t)^{\ve}} \, |\nabla v|^2  dxdt+\frac{1}{2\alpha^3}\int \sigma(x,t)^{-2\alpha} e^{2\alpha \sigma(x,t)^{\ve}} \, |\nabla^2v|^2dxdt\nonumber
\end{align}
 from which the lemma follows. 
 \end{proof}
 With  the estimate in \eqref{L-1} in hand, we now show that solutions to \eqref{main-diff-ineq} decay exponentially when they vanish to infinite order at the origin.
\begin{prop}\label{expdecay}
Let $u$ be a solution to  \eqref{main-diff-ineq} in $B_R \times (-T, T)$ such that $u$ vanishes to infinite order at $0$ in the sense of \eqref{vp}. Then  $u$ satisfies 
\begin{equation}
    \int_{B_s\times (-T/2,T/2)}u^2\lesssim e^{-\frac{C}{s^{2/3}}},~\text{as}~s\rightarrow0
\end{equation}
for some constant $C>0$.
\label{expo-prop}
\end{prop}
   \begin{proof}
Let $u$ be as in the statement of the proposition. For a given $r_1>0$ sufficiently small,  we let \[v(x,t)=u(x,t)\varphi(x)\eta(t)\] where $\varphi$ and $\eta$ are compactly supported functions which  are chosen as follows.

We let $\varphi$ such that
\begin{equation}
	\label{choice:phi}
	\varphi(x):=\begin{cases}
		1,~\text{if}~x\in B(0,r_1),\\
		0,~\text{if}~|x|>r_2
	\end{cases}
\end{equation}
where $r_2=  4\lambda^2 r_1< R$ with $\lambda$  as in \eqref{lam}.      As in  \cite{BGM,V}, we let  $T_1= \frac{3T}{4}$, $T_2= \frac{T}{2}$ and $\eta(t)$ be a smooth even function such that $\eta(t) \equiv 1$ when $|t| < T_2$, $\eta(t) \equiv 0$, when $|t| > T_1$. More precisely
\begin{equation}\label{choice:eta}
	\eta(t)= \begin{cases} 0\ \ \ \ \ \ \ \ \ -T\le t\le -T_1
		\\
		\exp \left(-\frac{T^3(T_2+t)^4}{(T_1 +t)^3(T_1-T_2)^4} \right)\ \ \ \ \ \ \ -T_1\le t \le -T_2,
		\\
		1,\ \ \ \ \ \ \   -T_2\le t \le 0.
	\end{cases}
\end{equation} 

  It is to be noted over here that from Remark \ref{H2} it follows that $\nabla u, \nabla^2 u, u_t$ vanishes to infinite order in the sense of \eqref{vp} as well and therefore by a standard   limiting argument, one can show that the Carleman estimate \eqref{L-1} can be applied to $v$ as above.  We denote the operator $\operatorname{div}(A(x,t) \nabla)$ by $\mathcal L$.  An easy calculation yields 
\[\mathcal{L}v-\partial_t v=\varphi \eta (\mathcal{L}u-\partial_t u)+2\langle A\nabla u, \nabla \varphi \rangle \eta+u(\eta ~\operatorname{div}(A\nabla\varphi)-\varphi\eta_t).\]

Also from \eqref{assump} and given our choice of $\varphi$ as in \eqref{choice:phi} it follows that
\begin{equation}\label{er}
\operatorname{div}(A\nabla\varphi) \leq C (|\nabla \varphi| + |\nabla^2 \varphi| ).\end{equation}
Using \eqref{er} we have
\begin{align}
	\label{L-2}
	&\int \sigma(x,t)^{-2\alpha}e^{2\alpha\sigma(x,t)^\ve} (\operatorname{div}(A(x,t)\nabla v)-\partial_t v)^2dxdt\\
	&\lesssim \int \sigma(x,t)^{-2\alpha}e^{2\alpha \sigma(x,t)^{\ve}}\varphi^2\eta^2(\operatorname{div}(A(x,t)\nabla u)-\partial_tu)^2+\int \sigma(x,t)^{-2\alpha}e^{2\alpha \sigma(x,t)^{\ve}}u^2\varphi^2\eta_t^2\notag\\
	&+\int \sigma(x,t)^{-2\alpha}e^{2\alpha \sigma(x,t)^{\ve}}\left[|\nabla u|^2|\nabla \varphi|^2+u^2|\nabla^2\varphi|^2+u^2|\nabla\varphi|^2\right]\eta^2\notag.  
\end{align}
 Now using \eqref{main-diff-ineq}  in \eqref{L-2} we find that $\int_{B_R\times (0,T)} \sigma(x,t)^{-2\alpha}e^{2\alpha\sigma(x,t)^\ve} (\operatorname{div}(A(x,t)\nabla v)-\partial_t v)^2dxdt$ can be estimated as \begin{align}
	\label{L-3}
	&\int_{B(0,R)\times (0,T)} \sigma(x,t)^{-2\alpha}e^{2\alpha\sigma(x,t)^\ve} (\operatorname{div}(A(x,t)\nabla v)-\partial_t v)^2dxdt\\
	&\lesssim \int \sigma(x,t)^{-2\alpha}|x|^{-4}e^{2\alpha \sigma(x,t)^{\ve}}\varphi^2\eta^2u^2+\int \sigma(x,t)^{-2\alpha}e^{2\alpha \sigma(x,t)^{\ve}}u^2\varphi^2\eta_t^2\notag\\
	&+\int_{\{r_1<|x|<r_2\}\times (-T_1,T_1)} \sigma(x,t)^{-2\alpha}e^{2\alpha \sigma(x,t)^{\ve}}\left[|\nabla u|^2|\nabla \varphi|^2+u^2(\Delta\varphi)^2+u^2|\nabla\varphi|^2\right]\eta^2\notag\\
	&\lesssim \int \sigma(x,t)^{-2\alpha-4}e^{2\alpha \sigma(x,t)^{\ve}}v^2+\int \sigma(x,t)^{-2\alpha}e^{2\alpha \sigma(x,t)^{\ve}}u^2\varphi^2\eta_t^2\notag\\
	&+\int_{\{r_1<|x|<r_2\}\times (-T_1,T_1)}\sigma(x,t)^{-2\alpha}e^{2\alpha \sigma(x,t)^{\ve}}\left[|\nabla u|^2|x|^{-2}+u^2|x|^{-4}+u^2|x|^{-2}\right]\eta^2.\notag
\end{align}
In \eqref{L-3} we used that $\nabla \varphi, \nabla^2 \varphi$ are supported in $\{r_1 < |x| < r_2\}$ and satisfies the following bounds
\begin{equation*}
|\nabla \varphi| \leq \frac{C}{|x|},\ |\nabla^2 \varphi|\leq \frac{C}{|x|^2}.\end{equation*}
Now by applying the Carleman estimate \eqref{L-1} to $v$, we obtain using \eqref{L-3} that the following holds
\begin{align}
	\label{est:f2}
&\alpha \int \sigma(x,t)^{-2\alpha-4} e^{2\alpha \sigma(x,t)^{\ve}} \, v^2  
+\frac{1}{2\alpha}\int \sigma(x,t)^{-2\alpha-2} e^{2\alpha \sigma(x,t)^{\ve}} \, |\nabla v|^2+\frac{1}{2\alpha^3}\int \sigma(x,t)^{-2\alpha} e^{2\alpha \sigma(x,t)^{\ve}} \, |\nabla^2v|^2 \\
&\lesssim\int \sigma(x,t)^{-2\alpha}|x|^{-4}e^{2\alpha \sigma(x,t)^{\ve}}\varphi^2\eta^2u^2+ \int \sigma(x,t)^{-2\alpha}e^{2\alpha \sigma(x,t)^{\ve}}u^2\varphi^2\eta_t^2-\alpha^3\int \sigma(x,t)^{-2\alpha- 4+\ve} e^{2\alpha \sigma(x,t)^{\ve}}  v^2\notag\\
&+\int_{\{r_1<|x|<r_2\}\times (-T_1,T_1)} e^{2\alpha \sigma(x,t)^{\ve}}\sigma(x,t)^{-2\alpha-4}u^2\eta^2
+\int_{\{r_1<|x|<r_2\}\times (-T_1,T_1)}e^{2\alpha \sigma(x,t)^{\ve}}\sigma(x,t)^{-2\alpha-2}|\nabla u|^2\eta^2\notag\\
&=:  \int \sigma(x,t)^{-2\alpha}|x|^{-4}e^{2\alpha \sigma(x,t)^{\ve}}\varphi^2\eta^2u^2 +    I_1+I_2+I_3+I_4\notag .
\end{align}
 We  then note that  since $\sigma(x,t) \sim |x|$, therefore if $\alpha$ is chosen large enough, then the integral  \\ $\int \sigma(x,t)^{-2\alpha}|x|^{-4}e^{2\alpha \sigma(x,t)^{\ve}}\varphi^2\eta^2u^2$ can be  absorbed into the term 
 $\alpha \int \sigma(x,t)^{-2\alpha-4} e^{2\alpha \sigma(x,t)^{\ve}} \, v^2$ which appears on the left hand side. 
 
 Now we estimate each $I_j$ separately. Let us first introduce some notations which will be used throughout the rest of the proof. 
\begin{align}\label{sigmas}
    \sigma_1(r_1, r_2, T):=\inf_{\{r_1<|x|<r_2\}\times (-T, T)}\sigma(x,t),~\text{and}~ \sigma_2(r_1, r_2, T):=\sup_{\{r_1<|x|<r_2\}\times (-T, T)}\sigma(x,t).
\end{align}  We first  note that 
it is easily seen that 
\begin{align}
	\label{est:I3} 
	I_3\lesssim  \sigma_1(r_1, r_2, T)^{-2\alpha-4}e^{2\alpha \sigma_2(r_1, r_2, T)^\ve}\int_{B(0,R)\times(-T,T)}u^2.
\end{align}
Similarly we have
\begin{align}\label{i4}
	I_4\lesssim \sigma_1(r_1, r_2, T)^{-2\alpha-2}e^{2\alpha \sigma_2(r_1, r_2, T)^\ve}\int_{\{r_1<|x|<r_2\}\times (-T_1,T_1)}|\nabla u|^2
\end{align}
Using the energy estimate in  Lemma~\ref{cac-bgg} we have 
\begin{align*}
&\int_{\{r_1<|x|<r_2\}\times (-T_1,T_1)}|\nabla u|^2\leq \frac{C}{r_1^2}\int_{B(0,R)\times (-T,T)}u^2
\end{align*}
which in particular implies
\begin{align}
\label{est:I4}
I_4\lesssim \sigma_1(r_1, r_2, T)^{-2\alpha-2}r_1^{-2}e^{2\alpha \sigma_2(r_1, r_2, T)^\ve}\int_{B(0,R)\times (-T,T)}u^2.
\end{align}
We now our attention to $I_1.$ To begin with, we break the integral into two parts as follows:

\begin{align*}
	I_1&= \int_{B_{r_1}\times (-T_1,T_1)}+\int_{\{r_1<|x|<r_2\}\times (-T_1,T_1)} \left(\sigma(x,t)^{-2\alpha}e^{2\alpha \sigma(x,t)^{\ve}}u^2\varphi^2\eta_t^2\right)\\
	&=:I_{11}+I_{12}
	\end{align*}
 The same technology as above, along with the estimate $|\eta_t|\lesssim \frac{1}{T}$ yields
 \begin{align}
 	\label{est:I11I13}
 	I_{12}\lesssim \sigma_1(r_1, r_2, T)^{-2\alpha}e^{2\alpha \sigma_2(r_1, r_2, T)^\ve} \int_{B(0,R)\times(-T,T)}u^2.
 \end{align}
We are  thus left with $I_{11}$, which will be estimated together with $I_2.$ \\
Recalling that by our choice  $\varphi(x)=1$ when $|x|<r_1$, and $supp(\eta_t)\subset (-T_1, -T_2)\cup (T_2,T_1)$ we observe that  
\begin{align*}
I_{11}+I_2\lesssim \int_{\Omega} \sigma(x,t)^{-2\alpha-4+\ve}e^{2\alpha \sigma(x,t)^{\ve}}u^2\eta^2\left(\sigma(x,t)^3\frac{\eta_t^2}{\eta^2}-\alpha^3\right)
\end{align*}
where $\Omega:=B_{r_1}\times [(-T_1, -T_2)\cup (T_2,T_1)].$ We now adapt some ideas from \cite{V} ( see also \cite{BGM}) to show that the following estimate holds\begin{align}
\label{est:I12I2}
I_{11}+I_2\lesssim \int_{B(0,R)\times (-T,T)}u^2.
\end{align}
First note that it is sufficient to prove \eqref{est:I12I2}  over the region $\Omega^{-} := B_{r_1} \times (-T_1, -T_2)$, since the other part can be handled by symmetry. Now, if $-T_1\le t\le -T_2$, note that $T_1 -T_2 = \frac T4$, $|T_2 + t| \le T_1 - T_2 = \frac T4$, and that $\frac 34 T \le 4T_1- 3T_2+t \le T$, it follows 
\begin{equation*}\label{y2}
\bigg|\frac{\eta_t}{\eta}\bigg| =\bigg|\frac{T^3 (T_2 + t)^3 (4T_1- 3T_2+t)}{ (T_1-T_2)^4 (T_1 +t)^4}\bigg| \leq   \frac{4 T^3}{|T_1 +t|^4}.
\end{equation*}
Consequently we obtain
\begin{align*}
&\int_{\Omega^{-}} \sigma(x,t)^{-2\alpha-4+\ve}e^{2\alpha \sigma(x,t)^{\ve}}u^2\eta^2\left(\sigma(x,t)^3\frac{\eta_t^2}{\eta^2}-\alpha^3\right)
 \leq \int_{\Omega^{-}} \sigma(x,t)^{-2\alpha-4+\ve}e^{2\alpha \sigma(x,t)^{\ve}}u^2\eta^2 \left(C \sigma(x,t)^3  \frac{T^6}{(T_1 +t)^8}  - \alpha^3\right).
\end{align*}
Trivially
\begin{align}
\label{eta-t-triv} 
\notag&\int_{\Omega^{-}} \sigma(x,t)^{-2\alpha-4+\ve}e^{2\alpha \sigma(x,t)^{\ve}}u^2\eta^2\left(\sigma(x,t)^3\frac{\eta_t^2}{\eta^2}-\alpha^3\right)\\
&\leq \int_{U} \sigma(x,t)^{-2\alpha-4+\ve}e^{2\alpha \sigma(x,t)^{\ve}}u^2\eta^2 \left(C \sigma(x,t)^3  \frac{T^6}{(T_1 +t)^8}  - \alpha^3\right),
\end{align}
where
\begin{equation}
    \label{eta-t-useful}
U:=\bigg\{(x,t)\in \Omega^-: \alpha^3\leq C \sigma(x,t)^3 \frac{T^6}{(T_1+t)^8}\bigg\}.     
\end{equation}
Thus it follows
\begin{align}\label{c7}
&\int_{\Omega^{-}} \sigma(x,t)^{-2\alpha-4+\ve}e^{2\alpha \sigma(x,t)^{\ve}}u^2\eta^2\left(\sigma(x,t)^3\frac{\eta_t^2}{\eta^2}-\alpha^3\right)\\
& \leq C \int_{U}  \sigma(x,t)^{-2\alpha-1+\ve}e^{2\alpha \sigma(x,t)^{\ve}}u^2 \eta \frac{\eta T^6}{(T_1 +t)^8}.\notag
\end{align}
Now in order to substantiate our claim we establish a bound from above for the quantity $$ \sigma(x,t)^{-2\alpha-1+\ve}e^{2\alpha \sigma(x,t)^{\ve}}\eta \frac{\eta T^6}{(T_1 +t)^8}$$ in $U$. Appealing to the exponential decay of $\eta$, at $t = - T_1$, see \eqref{choice:eta}, we obtain for $t \in (-T_1, -T_2)$, 
\begin{equation}\label{bn1}
\frac{\eta T^6}{(T_1+ t)^8} \leq 	C.
\end{equation}
Therefore, we will be done if we can prove that 
\begin{align*}
\notag &\sigma(x,t)^{-2\alpha-1+\ve}e^{2\alpha \sigma(x,t)^{\ve}}\eta\leq 1
\end{align*}
which, in view of the expression for $\eta$, is equivalent to proving 
\begin{align}
    \label{pmain-eq}  (2\alpha+1-\ve)\log\sigma(x,t)-2\alpha \sigma(x,t)^\ve+\frac{T^3(T_2+t)^4}{(T_1 +t)^3(T_1-T_2)^4}\geq 0
\end{align}
 for $(x,t) \in U$. Now in what follows, we prove \eqref{pmain-eq}. For that first note that, inside the region $U$, we have 
\begin{equation*}
\frac{T_1 +t}{T} \leq \left(\frac{C}{T^2}\right)^{1/8}\left(\frac{\sigma(x,t)}{\alpha}\right)^{3/8}
 \leq C \left(\frac{\sigma(x,t)}{\alpha}\right)^{3/8},
\end{equation*}
for some universal $C>0$ depending also on $T$. For sufficiently large  $\alpha$ we have 
\[
C \left(\frac{\sigma(x,t)}{\alpha}\right)^{3/8} \le C_1 \left(\frac{R}{\alpha}\right)^{3/8} \le \frac1{12}. 
\]Combining the above, we have
\begin{equation}\label{pp1-bgg}
\frac{T_1 +t }{T} \leq \frac1{12},
\end{equation}
in $U$, provided $\alpha$ is large enough. Also, $\frac T4 = T_1 - T_2 = T_1 + t + |T_2 + t|$, from \eqref{pp1-bgg} we conclude that we must have in $U$
\[
|T_2 + t| \geq \frac{T}{6}.
\]
It thus follows
\begin{align}\label{log-condn12} &(2\alpha+1-\ve)\log\sigma(x,t)-2\alpha \sigma(x,t)^\ve+\frac{T^3(T_2+t)^4}{(T_1 +t)^3(T_1-T_2)^4}\\
\notag &\geq \left(\frac{4}6\right)^{4} \left(\frac 2C\right)^{3/8} T^{3/4} \left(\frac{\alpha}\sigma\right)^{9/8}-(2\alpha+1-\ve)\log\frac{1}{\sigma(x,t)}-2\alpha \sigma(x,t)^\ve.
\end{align}
Now if we choose  $\alpha$ very large, in view of the fact that  the exponent $\alpha^{9/8}$ wins over  the linear term in $\alpha$, we infer that the last expression in \eqref{log-condn12} is non-negative proving \eqref{pmain-eq}.  Therefore, the estimate \eqref{est:I12I2} holds.

Therefore using \eqref{est:I3},\eqref{est:I4}, \eqref{est:I11I13} and \eqref{est:I12I2} in \eqref{est:f2}, we conclude that 
\begin{align}
\label{est:f3}
\alpha \int \sigma(x,t)^{-2\alpha-4} e^{2\alpha \sigma(x,t)^{\ve}} \, v^2 \lesssim \sigma_1(r_1, r_2, T)^{-2\alpha-4}e^{2\alpha \sigma_2(r_1, r_2, T)^\ve} \int_{B(0,R)\times (-T,T)}u^2.
\end{align}
We now estimate the left hand side of the above inequality in the following way
\begin{align*}
\alpha \int \sigma(x,t)^{-2\alpha-4} e^{2\alpha \sigma(x,t)^{\ve}} \, v^2&\geq \alpha \int_{B_{r_1/8\lambda}\times (-T_2,T_2)}\sigma(x,t)^{-2\alpha-4} e^{2\alpha \sigma(x,t)^{\ve}} \, v^2\\
&\geq \alpha \tilde \sigma_2(r_1/8\lambda, T)^{-2\alpha-4}\int_{B(0, r_1/8\lambda)\times (-T_2,T_2)}u^2
\end{align*}
where $$\tilde  \sigma_2(r_1/8\lambda,T):=\sup_{B_{r_1/8 \lambda}\times (-T, T)}\sigma(x,t)$$ and $\lambda$ is as in \eqref{lam}.
This, along with \eqref{est:f3} yields
\begin{align}\label{pun}
\alpha \int_{B_{ r_1/8\lambda} \times (-T_2,T_2)}u^2\lesssim  \left(\frac{\sigma_1(r_1, r_2, T)}{\tilde\sigma_2(r_1/8\lambda,T)}\right)^{-2\alpha-4}e^{2\alpha \sigma_2(r_1, r_2, T)^\ve} \int_{B(0,R)\times (-T,T)}u^2.
\end{align}
Now from  \eqref{sigmaest} we find
\begin{align}\label{pun1}
    \frac{\sigma_1(r_1, r_2, T)}{\tilde \sigma_2(r_1/8\lambda, T)}\geq \frac{Mr_1}{Nr_1/8\lambda}=8.
\end{align}
Using \eqref{pun1} in \eqref{pun} we then deduce the following inequality
\begin{align*}
    \int_{B_{ r_1/8\lambda} \times (-T_2,T_2)}u^2\lesssim 2^{-2\alpha} 4^{-2\alpha}e^{2\alpha \sigma_2(r_1, r_2, T)^\ve}.
\end{align*}
Next observe that \begin{align}&4^{-2\alpha}e^{2\alpha\sigma_{2}(r_1,r_2, T)^{\ve}}\leq e^{-2\alpha(2\log 2-(N r_2)^\ve)} \leq 1\\ & (\text{provided $r_1$ and consequently $r_2$ is small enough since $r_2= \lambda^2 r_1$}.)\notag\end{align}   Therefore we obtain\begin{align}\label{res1}
    \int_{B_{r_1/8 \lambda} \times (-T_2,T_2)}u^2\lesssim 2^{-2\alpha}.
\end{align}
At this point, in view of  Lemma \ref{carlemanL}, by letting $\alpha \sim \left(\frac{1}{r_1} \right)^{2/3}$, we find that the conclusion of the lemma follows from \eqref{res1}. \end{proof}

\subsection{Carleman estimate II}
Proposition \ref{expdecay} shows that solutions to \eqref{main-diff-ineq}  decay exponentially at the origin when the vanish to infinite order in the sense of \eqref{vp}. We now show  by means of a parabolic generalization of a Carleman estimate due to  Regbaoui in \cite{Reg}  which uses strictly convex  weights  that non-trivial solutions to  \eqref{main-diff-ineq} in fact decay less than exponentially which would then lead to a contradiction  and would thus establish Theorem \ref{thm1-bgg}.

We start by recalling a few  notations. As before,  $g=(g_{ij}(x, t))$ will denote the inverse of the coefficient matrix $A(x, t)$ and $\sigma(x,t)$ is as in \eqref{weight-log-estimate}.  For notational convenience, we will often denote $\sigma(x,t)$ by $\sigma$ whenever the context is clear.  Before proceeding further, we would like to alert the reader that for notational convenience,  we set 
$$ \mathcal L= \operatorname{div}(A(x,t)\nabla).$$

The following vector field will play a pervasive role in our analysis.
\begin{align}
\label{vector-field}
Z=\sigma\sum_{i, j=1}^{n} a_{ij}\partial_{i}\sigma\partial_{j}. 
\end{align}
We now collect some basic properties related to  $Z$ and $\sigma$ which is easily verified using \eqref{assump} and the fact that the matrix valued function $(g_{ij})$ is the inverse of $(a_{ij})$.

\begin{prop}
\label{useful-estimates}
The following holds true:
\begin{enumerate}[i)]
    \item $|<ZA\nabla u, \nabla u>| \leq C \sigma |\nabla u|^2$.
		\item $|[\partial_i, Z]u -\partial_i u| \leq C \sigma |\nabla u|$,\ \ \ \ $i=1,...,N$.
		\item  $\partial_i Z_j = \delta_{ij} + O(|x|).$
		\item $Z\sigma = \sigma + O(|x|^2).$
\end{enumerate}
\end{prop}
\begin{proof}
We only prove (ii), (iii) and (iv) since  (i) is easily seen to be true.\\
\noindent\textit{Proof of (ii):} In view of the definition of $Z$, we see that 
\begin{align*}
	[\partial_i, Z]u=\partial_{i}Zu-Z\partial_{i}u
	=\sum_{k,l}\partial_{i}(\sigma a_{kl}\partial_k\sigma)\partial_lu+\sigma\sum_{k,l}a_{kl}\partial_k\sigma[\partial_{i},\partial_{j}]u.
\end{align*}
But as $[\partial_i,\partial_j]u=0$, the last term in the above expression vanishes. 

Observe now that
\begin{align}
	\label{derivative-sigma}
	\partial_{k}\sigma=\frac{\sum_{m} g_{km}(0, t)x_{m}}{\sigma}.
\end{align}
Now, using \ref{derivative-sigma} and $a_{kl}=a_{kl}(0,t)+b_{kl} (x,t)$, we observe that 
\begin{align}\label{hjj}
&\sum_{k, l} \partial_i ( \sigma a_{kl} \partial_k \sigma) \partial_l u= \sum_{k,l,m} \delta_{im} a_{kl}(0,t) g_{km}(0,t) \partial_l u + \sum_{k, m, l} \partial_i (b_{kl} g_{km}(0,t)x_m) \partial_l u.
\end{align}
Now using 
\[
\sum_{k} a_{kl} (0,t) g_{km}(0,t)=\delta_{lm}\ \text{(since $(g_{ij})$ is the inverse of $(a_{ij})$.)}
\]
it follows that 
\begin{equation}\label{hj1}
\sum_{k,l,m} \delta_{im} a_{kl}(0,t) g_{km}(0,t) \partial_l u= \partial_i u.
\end{equation}
Also since $b_{ij}(x,t) =O(|x|)$, it follows
\begin{equation}\label{hj2}
\left| \sum_{k, m, l} \partial_i (b_{kl} g_{km}(0,t)x_m) \partial_l u \right| \leq C|x| |\nabla u|.
\end{equation}
Using \eqref{hj1} and \eqref{hj2} in \eqref{hjj}, (ii) is seen to hold.

\medskip

\noindent\textit{Proof of (iii):} 
 Using standard summation convention we have $$Z_{j}=a_{ij}(x, t)g_{im}(0, t)x_{m}$$ for any $1\leq j\leq n$.  As before, by writing $a_{ij}(x,t)= a_{ij}(0,t) + b_{ij}(x,t)$, we have
 \begin{equation}\label{zj}
 Z_j = \delta_{jm} x_m + b_{ij} g_{im}(0,t) x_m= x_j + b_{ij} g_{im}(0,t) x_m.\end{equation}
  
Noting the fact that $b_{ij}=O(|x|)$, it is  thus straightforward to see from \eqref{zj} that $\partial_{i} Z_j=\delta_{ij}+O(|x|)$.\\
\medskip
\noindent\textit{Proof of (iv):} We have 
$$Z\sigma=\frac{a_{ij}g_{il}x_{l}g_{jm}x_{m}}{\sigma}.$$
Again  by writing $a_{ij}=a_{ij}(0, t)+b_{ij}$ and  using that  $b_{ij}=O(|x|)$, an easy calculation leads to the following

\begin{align*}
Z\sigma&=\frac{g_{jm}x_{j}x_{m}}{\sigma}+\frac{\sum_{i, j, l}b_{ij}g_{il}x_{l}g_{jm}x_{m}}{\sigma}\\
&=\sigma+O(|x|^2).
\end{align*}
 
\end{proof}

We also require  the following Rellich's identity (see for instance \cite{Ne, PWe}). Given a $C^{0,1}$ vector  field $G$, we have
\begin{align}\label{Rellich-bgg}
& \int_{\partial B_R} \langle A \nabla u, \nabla u\rangle \, \langle G,\nu \rangle  = 2 \int_{\partial B_R} a_{ij} \partial_i u \langle \partial_j,\nu \rangle Gu 
\\
& -  2\int_{B_R} a_{ij} (\operatorname{div} \partial_i) \partial_j u Gu  - 2 \int_{B_R} a_{ij} \partial_i u [\partial_j, G]u
\notag\\
& + \int_{B_R} \operatorname{div} G \langle A\nabla u, \nabla u \rangle +  \int_{B_R} \langle (GA)\nabla u, \nabla u \rangle  - 2 \int_{B_R} Gu\, \partial_i(a_{ij} \partial_j u),
\notag
\end{align}
where $ \nu $ denotes  the outer unit normal to $B_R$. The following application of  such a  Rellich's identity will be required in our setting.

\begin{lemma}
\label{Rellich-BGG}
Denoting $\sigma(\cdot, t)$ by $\sigma(\cdot)$, the following holds true for $u \in C^{\infty}_{0}(B_R \setminus \{0\})$
\begin{align}
\label{Rellich-application}
&2\int_{B_{R}} \sigma^{-n+2} (-\log \sigma) Zu\,\mathcal{L}u=-\int_{B_{R}}\sigma^{-n+2} \langle A \nabla u, \nabla u\rangle\\
&+2\int_{B_{R}} \sigma^{-n}((n-2)(-\log \sigma)+1) (Zu)^{2} +  O(1) \int_{B_R} \sigma^{-n+3} (-\log \sigma) |\nabla u|^2.\notag \\
\end{align}
\end{lemma}
\begin{proof}
With $G=\sigma^{-n+2}(-\log \sigma) Z$, by applying \eqref{Rellich-bgg} we find
\begin{align}\label{Rellich-intermediate}
&2\int_{B_{R}} \sigma^{-n+2} (-\log \sigma) Zu\,\mathcal{L}u=\int_{B_R} \operatorname{div}(\sigma^{-n+2} (-\log \sigma) Z) \langle A\nabla u, \nabla u \rangle\\
&- 2 \int_{B_R} a_{ij} \partial_i u [\partial_j, \sigma^{-n+2} (-\log \sigma) Z]u+  \int_{B_R} \sigma^{-n+2} (-\log \sigma) \langle ZA\nabla u, \nabla u \rangle\notag.
\end{align}
  From \eqref{Rellich-intermediate}  we obtain
  
  \begin{align}
\label{Rellich-application1}
&2\int_{B_{R}} \sigma^{-n+2} (-\log \sigma) Zu\,\mathcal{L}u=-\int_{B_{R}}\sigma^{-n+2}((n-2)(-\log \sigma)+1) \langle A \nabla u, \nabla u\rangle\\
&+ (n-2)\int_{B_{R}}\sigma^{-n+2}(-\log \sigma) \langle A \nabla u, \nabla u\rangle
-2\int_{B_{R}} \sigma^{-n+2} (-\log \sigma) a_{ij}\partial_{i }u \left([\partial_{j}, Z]u-\partial_{j}u\right)\notag\\
&+2\int_{B_{R}} \sigma^{-n}((n-2)(-\log \sigma)+1) (Zu)^{2} +  O(1) \int_{B_R} \sigma^{-n+3} (-\log \sigma) |\nabla u|^2,\notag \\
\end{align}

  once we incorporate the following:

\begin{align}\label{iden1}
&\operatorname{div}(\sigma^{-n+2} (-\log \sigma) Z)=-\sigma^{-n+2}(1+(n-2)(-\log\sigma))+\sigma^{-n+2}(-\log \sigma) n+ O(\sigma^{-n+3}) (-\log \sigma)\ \text{and} \\
& \int_{B_{R}} \sigma^{-n+2} (-\log\sigma)\langle (ZA)\nabla u, \nabla u \rangle= O(1) \int \sigma^{-n+3} (-\log \sigma) |\nabla u|^2,\ \notag
\end{align}    

which follows from  (i),  (iii) and (iv) of Proposition \ref{useful-estimates} and  also by using
\begin{align}\label{iden2}
 &[\partial_j, \sigma^{-n+2} (-\log \sigma) Z]u=\sigma^{-n+2}(-\log\sigma)\left([\partial_{j}, Z]u-\partial_{j}u\right)\\
 &+\sigma^{-n+2}(-\log\sigma)\partial_{j}u-\left((n-2)(-\log\sigma)+1\right)\sigma^{-n+1} \partial_{j}\sigma\,Zu.\notag
\end{align}
Finally by using (ii) of Proposition \ref{useful-estimates}, we find that \eqref{Rellich-application} follows from \eqref{iden2}.

\end{proof}

We now state and prove  the following parabolic version of Regbaoui's  Carleman estimate. 
\begin{thrm} \label{hardy-BGG}
	There exists  universal $C>0$ such that for every $\beta >0$ sufficiently large,  $R_0$ sufficiently small, $u \in C^{\infty}_{0}((B_R  \setminus \{0\}) \times (-T, T))$  for $R\leq R_0$, one has
	\begin{align}\label{har1-BGG}
	&\beta^3 \int_{B_R} \sigma^{-n} e^{\beta (\log \sigma)^2}u^2 +  \beta \int_{B_R}  \sigma^{-n+2} e^{\beta (\log \sigma)^2} \langle A \, \nabla u, \nabla u\rangle\\
	&\leq C  \, \int \sigma^{-n+4} e^{\beta (\log \sigma)^2} (\La  u-\Dt u)^2. \notag
	\end{align}
\end{thrm}

Theorem \ref{hardy-BGG} is quintessential in proving that non-trivial solutions to \eqref{main-diff-ineq} decay less than exponentially. Proof of Theorem~\ref{hardy-BGG} divided into some intermediate results. The first such result is as follows.\begin{thrm} \label{hardy-BGG-semi-final}
	There exists  universal $C>0$ such that for every $\beta >0$ sufficiently large,  $R_0$ sufficiently small, $u \in C^{\infty}_{0}((B_R  \setminus \{0\}) \times (-T, T))$  for $R\leq R_0$, one has	\begin{align}
\label{semi-final-BGG} \int \sigma^{-n+4} \, e^{\beta(\log \sigma)^2}(\La
	 u-\Dt u)^2&\geq \frac{59}{10} \beta^3 \int \sigma^{-n} (\log\sigma)^2 v^2+\frac{15}{4} \beta\int \sigma^{-n} (Zv)^{2}\\
	&\notag-\frac{5}{2} \beta \int \sigma^{-n+2} \langle A \nabla v, \nabla v\rangle,
	\end{align}
	where  $v= e^{\frac{\beta}{2} (\log \sigma)^{2}}   \, u$.\end{thrm}
\begin{proof}
	It follows 	\begin{eqnarray*}
		\La u=\La v (e^{-\frac{\beta}{2} (\log \sigma)^{2}} )+2 a_{ij} \, \partial_jv \partial_i(e^{-\frac{\beta}{2} (\log \sigma)^{2}} )+\La (e^{-\frac{\beta}{2} (\log \sigma)^{2}} ) v
	\end{eqnarray*}
	
	As before, letting $a_{ij}(x,t)=a_{ij}(0, t)+b_{ij}(x,t)$, we first compute  the term $ \sum_{i,j} a_{ij}(0, t) (e^{-\frac{\beta}{2} (\log \sigma)^{2}} )_{ij}$ explicitly.  To do so, one important ingredient is the following
$$(e^{-\frac{\beta}{2} (\log \sigma)^{2}} )_{i}=e^{-\frac{\beta}{2} (\log \sigma)^{2}} (-\beta \log\sigma)\frac{g_{il}x_{l}}{\sigma^{2}}.$$
Upon differentiating once more, we obtain
\begin{align*}
a_{ij}(0, t) (e^{-\frac{\beta}{2} (\log \sigma)^{2}} )_{ij}&=\frac{e^{-\frac{\beta}{2} (\log \sigma)^{2}} (\beta \log\sigma)^2}{\sigma^4} a_{ij}(0, t) g_{jm}x_{m} g_{il}x_{l}\\
&+\frac{(-\beta)e^{-\frac{\beta}{2} (\log \sigma)^{2}}}{\sigma^4}a_{ij}(0, t) g_{jm}x_{m} g_{il}x_{l}\\
&+\frac{(-\beta \log\sigma)e^{-\frac{\beta}{2} (\log \sigma)^{2}}}{\sigma^2}a_{ij}(0, t) g_{il}\delta_{jl}\\
&-2\frac{(-\beta \log\sigma)e^{-\frac{\beta}{2} (\log \sigma)^{2}}}{\sigma^4}a_{ij}(0, t) g_{jm}x_{m} g_{il}x_{l}.
\end{align*}	
Now using the fact that $a_{ij}(0, t) g_{jm}x_{m} g_{il}x_{l}=g_{lm} x_{m} x_{l}=\sigma^2$, we obtain
\begin{align}
a_{ij}(0, t) (e^{-\frac{\beta}{2} (\log \sigma)^{2}} )_{ij}=\ee \sigma^{-2}\left((\beta \log\sigma)^2-\beta-(n-2)\beta\log\sigma\right).
\end{align}

We thus  obtain
\begin{align}
\La u= &\La v (e^{-\frac{\beta}{2} (\log \sigma)^{2}} ) +2 \sigma^{-2}(-\beta \log\sigma)\ee\, Zv\\
&\notag +\ee \sigma^{-2}\left((\beta \log\sigma)^2-\beta-(n-2)\beta\log\sigma\right)v\\
&+\left[\partial_{i}b_{ij}\ee (-\beta \log\sigma)\sigma^{-1}\partial_{j}\sigma \right]v\notag+\left[b_{ij}\left(-\beta (\log \sigma) \ee\sigma^{-1} \, \right)\partial_i \partial_j \sigma\right]	v \nonumber\\
&+\left[b_{ij}\left(\beta \sigma^{-2} \ee  \, \left(-1+(\log \sigma)+\beta (\log \sigma)^2\right)\right)\partial_i \sigma \partial_j \sigma\right] v\notag.
\end{align}
Similarly we also have
\begin{align*}
\partial_{t}u=\ee\left[\Dt v-v \beta \log\sigma\frac{\Dt \sigma}{\sigma}\right].
\end{align*}
Now by writing 
\begin{align}
&(\La u-\Dt u)e^{\frac{\beta}{2}(\log\sigma)^{2}}=A+B,\,\text{where,}\\
&A:=2 \sigma^{-2}(-\beta \log\sigma)\, Zv-\Dt v,\\
&\text{and,}\,\, B:=\La v +\sigma^{-2}\left((\beta \log\sigma)^2-\beta-(n-2)\beta\log\sigma\right)v\\
&+\left[\partial_{i}b_{ij} (-\beta \log\sigma)\sigma^{-1}\partial_{j}\sigma \right]v\notag+\left[b_{ij}\left(-\beta (\log \sigma) \sigma^{-1} \, \right)\partial_i \partial_j \sigma\right]	v \nonumber\\
&+\left[b_{ij}\left(\beta \sigma^{-2}   \, \left(-1+(\log \sigma)+\beta (\log \sigma)^2\right)\right)\partial_i \sigma \partial_j \sigma\right] v +v \beta \log\sigma \frac{\partial_t \sigma}{\sigma}\notag
\end{align}
and then by using the algebraic inequality $(A+B)^2 \geq A^2 + 2AB$ we obtain

	\begin{align}\label{c2-BGG}
	&\int \sigma^{-n+4} \, e^{\beta(\log \sigma)^2}(\La
	 u-\Dt u)^2\geq \int \sigma^{-n+4} (A^2+2AB)\\
	& = \int \sigma^{-n+4} [2 \sigma^{-2}(-\beta \log\sigma)\, Zv-\Dt v]^{2}+\int 4\beta \sigma^{-n+2} (-\log \sigma) \, Zv \La v  \notag\\
	&
	+ \int 4 \beta \,  \sigma^{-n+2} \, (-\log \sigma) \bigg[\left[\sigma^{-2}\left((\beta \log \sigma)^2-\beta-(n-2) \beta \log \sigma \right)\right] \notag\\
	&+\left[b_{ij}\left(\beta (-\log \sigma) \sigma^{-1} \, \right)\partial_i \partial_j \sigma\right] \notag\\
	&+\left[b_{ij}\left(\beta \sigma^{-2}  \, \left(-1+(\log \sigma)+\beta (\log \sigma)^2\right)\right)\partial_i \sigma \partial_j \sigma\right]\bigg] Zv \cdot	v  \notag\\
	&+4\beta\int \sigma^{-n+1}\left[\partial_i b_{ij} \cdot \beta (-\log \sigma)\,  \partial_j\sigma\right]Zv\cdot v- 4\beta^{2} \int\sigma^{-n+4}(\log\sigma)^2 \sigma^{-2} Zv. v \frac{\partial_t \sigma}{\sigma}\notag\\
	&-2\int \sigma^{-n+4} \La v\, \Dt v-2\int \sigma^{-n+2}\left((\beta \log\sigma)^2-\beta-(n-2)\beta\log\sigma\right)v \,\Dt v\notag\\
	&-2\int \sigma^{-n+3}\left[\partial_{i}b_{ij} (-\beta \log\sigma)\partial_{j}\sigma \right]v\, \Dt v-2\beta\int \sigma^{-n+3} b_{ij} (-\log\sigma) (\partial_{i} \partial_{j}\sigma)\, v\, \Dt v\notag\\
	&-2\beta \int \sigma^{-n+2} b_{ij}[\left(-1+(\log \sigma)+\beta (\log \sigma)^2\right)]\partial_{i} \sigma\, \partial_{j}\sigma\, v\, \Dt v\notag\\
	&-2\beta \int \sigma^{-n+3} (\log\sigma)v \Dt v \Dt \sigma\notag=:\sum_{j=1}^{11}\mathcal{I}_{j}.
	\end{align}
We now examine each $\mathcal{I}_{j}$ separately.
\medskip
\paragraph{\underline{Estimate for $\mathcal{I}_2$}:}  Using Lemma~\ref{Rellich-BGG} we have

\begin{align}
\label{i2} &4\beta\int_{B_{R}} \sigma^{-n+2} (-\log \sigma) Zv\,\mathcal{L}v\\
\notag &\notag\geq 4 \beta\int \sigma^{-n}((n-2)(-\log \sigma)+1) (Zv)^{2}-2\beta\int \sigma^{-n+2} \langle A \nabla v, \nabla v\rangle\\
&-4c\beta \int \sigma^{-n+3}(-\log \sigma) \langle A \nabla v, \nabla v\rangle\notag\\
& \notag\geq 4 \beta\int \sigma^{-n}((n-2)(-\log \sigma)+1) (Zv)^{2}-\frac{9}{4} \beta \int \sigma^{-n+2} \langle A \nabla v, \nabla v\rangle,\notag
\end{align}
where the penultimate inequality is a consequence of $c\,\sigma^{-n+3}(-\log\sigma)\leq \frac{1}{16}\sigma^{-n+2}$ which holds for $\sigma$ small enough which in turn can be ensured by taking $R_0$ small enough. This completes the estimate for $\mathcal{I}_2$.
\medskip
\paragraph{\underline{Estimate for $\mathcal{I}_3$}:} Write $\mathcal{I}_3=\sum_{k=1}^{3}\mathcal{I}_{3}^{k}.$ Let us begin by estimating the term \begin{align}
\notag\mathcal{I}_{3}^{1}&:=\int 4 \beta \,  \sigma^{-n} \, (-\log \sigma) \left[\left((\beta \log \sigma)^2-\beta-(n-2) \beta \log \sigma \right)\right] (Zv)\, v\\
\notag&=\int 4 \beta \, [(\beta \log \sigma)^2-\beta-(n-2) \beta \log \sigma] (-\log \sigma) \sigma^{-n} Z \bigg(\frac{v^2}{2}\bigg)\\
&=  4 \int \beta^3 \operatorname{div}[(\log \sigma)^3\sigma^{-n} Z] \frac{v^2}{2}-4 \int \beta^2 \operatorname{div}((\log \sigma)\sigma^{-n} Z) \frac{v^2}{2}\label{i3-1}\\
&-4 \int \beta^2(n-2) \operatorname{div}((\log \sigma)^2\sigma^{-n} Z) \frac{v^2}{2},\notag
\end{align}
where we have used integration by parts. Observe
\begin{align*}
4 \int \beta^3 \operatorname{div}[(\log \sigma)^3\sigma^{-n} Z ] \frac{v^2}{2}&=2\beta^{3}\int \left(3 (\log\sigma)^2-n(\log\sigma)^3\right)    \sigma^{-n-1}Z\sigma\, v^2\\
&+2\beta^3 \int (\log \sigma)^3 \sigma^{-n} \operatorname{div}(Z)\, v^2.
\end{align*}
Since $|\operatorname{div}Z-n|\leq c\sigma$, we obtain the following estimate
\begin{align}
\label{i3-11}
4 \int \beta^3 \operatorname{div}[(\log \sigma)^3\sigma^{-n} Z ] \frac{v^2}{2}\geq 6\beta^3 \int \sigma^{-n}(\log\sigma)^{2} v^2-c \beta^{3}\int \sigma^{-n+1}(\log \sigma)^3 v^2.  
\end{align}
Arguing similarly we obtain 
\begin{align}
\label{i3-12}
	 &-4 \int \beta^2 \operatorname{div}((\log \sigma)\sigma^{-n} Z) \frac{v^2}{2} -4 \int \beta^2(n-2) \operatorname{div}((\log \sigma)^2\sigma^{-n} Z) \frac{v^2}{2} \\
	& \notag\geq2\beta^2	 \int  \sigma^{-n} [2(n-2)(-\log \sigma)-1]  \, v^2- C \int  \sigma^{-n+1}  [\beta^2(-\log \sigma) + 4 \beta^2(n-2)(\log \sigma)^2 ] v^2.
\end{align}	
\begin{align}\label{i32}
	\mathcal{I}_{3}^{2}&:= 4 \beta \int  \sigma^{-n+1} \left[ b_{ij} (-\log \sigma)  \partial_i \partial_j \sigma \right] Z\bigg(\frac{v^2}{2} \bigg)=4 \beta\int \operatorname{div}\bigg(\sigma^{-n+1}\left[b_{ij}(\log \sigma)\,\partial_i \partial_j \sigma Z\right] \bigg)\frac{v^2}{2}\\
	\notag&=4\beta\int \bigg[(n-1)\sigma^{-n}\,Z\sigma\, (b_{ij} \partial_i \partial_j \sigma)(-\log\sigma)+\sigma^{-n}Z\sigma\, (b_{ij} \partial_i \partial_j \sigma)\\
	\notag &+\sigma^{-n+1}(\log\sigma) Z(b_{ij} \partial_i \partial_j \sigma)+\sigma^{-n+1}(b_{ij}(\log \sigma)\,\partial_i \partial_j \sigma)(\operatorname{div}Z)\bigg]\\
	&\notag \geq - C \beta \int \sigma^{-n+1} (-\log \sigma) v^2.
			\end{align}
 Similarly
\begin{align}
	\notag &\mathcal{I}_{3}^{3}:= 4 \beta^2 \int \sigma^{-n} (-\log \sigma) (-1+(\log \sigma)+\beta (\log \sigma)^2)b_{ij} \partial_i \sigma \partial_j \sigma  Z \bigg(\frac{v^2}{2} \bigg)\\ & = - 2 \beta^2 \int \operatorname{div} ( \sigma^{-n} (-\log \sigma) (-1+(\log \sigma)+\beta (\log \sigma)^2)b_{ij} \partial_i \sigma \partial_j \sigma  Z) v^2\notag 	\\
	& \geq - C\beta^2 \int \sigma^{-n+2} (-\log \sigma) (-1+(\log \sigma)+\beta (\log \sigma)^2) v^2- C\beta^2 \int \sigma^{-n+1} v^2\label{i33} \\
&- C\beta^3 \int \sigma^{-n+1} (\log \sigma)^2 v^2. \notag
		\end{align}
Therefore, combining estimates \eqref{i3-1}, \eqref{i3-11}, \eqref{i3-12}, \eqref{i32}, and \eqref{i33}, we deduce the following inequality for $R_0$ small enough and $\beta$ sufficiently large
\begin{align}\label{i3}
\mathcal{I}_{3}\geq \frac{599}{100} \, \beta^3 \int \sigma^{-n}( \log \sigma)^2 v^2.    
\end{align}
\medskip
\paragraph{\underline{Estimate for $\mathcal{I}_4$}:} An application of the Young's inequality $ab\leq \frac{\epsilon}{2} a^2+\frac{1}{2\epsilon}b^2,$ together with the fact that $\sum_{i, j} |\partial_{i}b_{ij} \partial_{j}\sigma|\leq c$ gives \begin{align}\label{i4}
		&\int 4 \beta \,  \sigma^{-n+1} \left[\partial_i b_{ij} \cdot \beta (-\log \sigma) \partial_j \sigma\right] Zv\cdot v 
		\\
		& \geq -  c\beta \int \sigma^{-n+1} ( Zv)^2 (\log(\sigma))^2  -\beta^{3}  \int \sigma^{-n+1}  v^2.\notag
	\end{align}

\medskip
\paragraph{\underline{Estimate for $\mathcal{I}_5$}:} A similar application of the Young's inequality implies for all small enough $R_0$ \begin{align}
\label{i5}\mathcal{I}_5&:=-\int 4\beta^{2}\sigma^{-n+2}(\log\sigma)^2 Zv. v\\
&\geq -2\beta \int \sigma^{-n+2} (\log\sigma)^2 (Zv)^2-2\beta^3\int \sigma^{-n+2} (\log(\sigma))^2 v^2.\notag
\end{align}
\medskip
Thus from  \eqref{i2}, \eqref{i3}, \eqref{i4} and \eqref{i5}, we obtain
\begin{align}
\label{final-estimate-Zv} & \sum_{i=2}^5 \mathcal I_i\\
&\geq 4 \beta\int \sigma^{-n}((n-2)(-\log \sigma)) (Zv)^{2}+\frac{31}{8} \beta\int \sigma^{-n} (Zv)^{2}\notag\\
&\notag+\frac{598}{100} \beta^3 \int \sigma^{-n} (\log\sigma)^2 v^2-\frac{9}{4} \beta \int \sigma^{-n+2} \langle A \nabla v, \nabla v\rangle,
\end{align}
provided we choose $\beta$ sufficiently large and $R_0$  small enough. Over here we would like to highlight the fact  that the term $4 \beta\int \sigma^{-n}((n-2)(-\log \sigma)) (Zv)^{2}$ in \eqref{final-estimate-Zv} above  will be very helpful for us in dealing with some subsequent negative terms.
\medskip\\
Now we start estimating terms involving $v_t$. 
\medskip
\paragraph{\underline{Estimate for $\mathcal{I}_6$}:}
Recall that 
\begin{align}
\notag\mathcal{I}_6&:=-2\int \sigma^{-n+4} \La v\, \Dt v\\
\notag&=2\int A\nabla v\cdot \nabla(\Dt v\cdot \sigma^{-n+4})\\
\notag &=2\int  \sigma^{-n+4} A\nabla v\cdot \nabla {\Dt v}-2(n-4)\int \langle A\nabla v, \nabla  \sigma\rangle \sigma^{-n+3} \Dt v\\
\notag &=-\int \langle(\Dt A) \nabla v, \nabla v\rangle \sigma^{-n+4}+2(n-4)\int \langle A\nabla v, \nabla v\rangle \sigma^{-n+3} \Dt \sigma\\ \notag &-2(n-4)\int \sigma^{-n+2} Z(v)\cdot\Dt v\\
\label{i6}&\geq- C\int |\nabla v|^2 \sigma^{-n+4}+2C(n-4)\int |\nabla v|^2 \sigma^{-n+4} \\ \notag &-2(n-4)\int \sigma^{-n+2} Z(v)\cdot\Dt v, 
\end{align}
where we used that $|\Dt \sigma| \leq C \sigma$.
\medskip
\paragraph{\underline{Estimate for $\mathcal{I}_7$}:} Again considering fact that $\Dt\sigma\simeq  O(\sigma)$, using integration by parts we obtain 
\begin{align}
\notag\mathcal{I}_7&:=-2\int \sigma^{-n+2}\left((\beta \log\sigma)^2-\beta-(n-2)\beta\log\sigma\right)v \,\Dt v\notag\\
\notag&=\int \Dt\left(\sigma^{-n+2}\left((\beta \log\sigma)^2-\beta-(n-2)\beta\log\sigma\right)\right) v^2\\
\notag &=\int \bigg[(-n+2) \sigma^{-n+1} \Dt \sigma\left((\beta \log\sigma)^2-\beta-(n-2)\beta\log\sigma\right)\\
\notag &+\sigma^{-n+2}\left(2\beta^2 \log\sigma \frac{1}{\sigma} \Dt \sigma-\beta(n-2)\frac{1}{\sigma} \Dt \sigma\right)\bigg]\, v^2\\
&\geq - C\beta^2 \int \sigma^{-n+2} ( \log\sigma)^2v^2,\label{i7}
\end{align}
provided $\beta$ is sufficiently large and $R_0$ is small enough. 
\medskip
\paragraph{\underline{Estimate for $\mathcal{I}_8$}:} Using Young's inequality, we obtain that $\mathcal I_8$ can be estimated as
\begin{align}
\label{i8}
\mathcal{I}_8&:=-2\int \sigma^{-n+3}\left[\partial_{i}b_{ij} (-\beta \log\sigma)\partial_{j}\sigma \right]v\, \Dt v\\
\notag&\geq -c\beta^{3} \int \sigma^{-n+1} (\log\sigma)^2 v^2-\frac{1}{4\beta} \int \sigma^{-n+5} (\Dt v)^2,
\end{align}
where we have used the fact $\sum_{i, j} |\partial_{i}b_{ij} \partial_{j}\sigma|= O(1)$. 

\medskip

\paragraph{\underline{Estimate for $\mathcal{I}_9$}:} Using  $|b_{ij} (\partial_{i}\partial_{j}\sigma)|= O(1)$ and Young's inequality we have
\begin{align}
\label{i9}\mathcal{I}_9&:=-2\beta\int \sigma^{-n+3} b_{ij} (-\log\sigma) (\partial_{i} \partial_{j}\sigma)\, v\, \Dt v\\
\notag&\geq -c_1\beta^{3} \int \sigma^{-n+1} (\log\sigma)^2 v^2-\frac{1}{4\beta} \int \sigma^{-n+5} (\Dt v)^2.
\end{align}

\medskip
\paragraph{\underline{Estimate for $\mathcal{I}_{10}$}:}
\begin{align}
\mathcal{I}_{10}&:=-2\beta \int \sigma^{-n+2} b_{ij}[\left(-1+(\log \sigma)+\beta (\log \sigma)^2\right)]\partial_{i} \sigma\, \partial_{j}\sigma\, v\, \Dt v
\end{align}
We first handle the first two terms in the expression for $\mathcal I_{10}$.  It is easy to observe using $|b_{ij}| =O(\sigma)$ and Young's inequality  that 
\begin{align}
\label{i10-1}
&-2\beta \int \sigma^{-n+2} b_{ij}[\left(-1+(\log \sigma)\right)]\partial_{i} \sigma\, \partial_{j}\sigma\, v\, \Dt v\\
&\geq  -c_2\beta^{3} \int \sigma^{-n+1} (-\log\sigma)^{2} v^2-\frac{1}{4\beta} \int \sigma^{-n+5} (\Dt v)^2. \notag
\end{align}

Now using integration by parts and the fact that $|\Dt( \partial_{i}\sigma\partial_{j}\sigma)|\leq c$,
\begin{align}
\label{i10-2}
&-2\beta^{2} \int\sigma^{-n+2} b_{ij} (\log \sigma)^2\partial_{i} \sigma\, \partial_{j}\sigma\, v\, \Dt v\\
&\notag =\beta^{2} \int\Dt\{\sigma^{-n+2} b_{ij} (\log \sigma)^2\partial_{i} \sigma\, \partial_{j}\sigma\}\, v^2 \\
&\notag\geq -c_3\beta^{2}\int \sigma^{-n+2}(\log\sigma)^2 v^2.
\end{align}
\medskip
\paragraph{\underline{Estimate for $\mathcal{I}_{11}$}:}
Using $|\Dt \sigma|=O(\sigma)$, we  simply observe
\begin{align}
\label{i-11}&-2\beta \int \sigma^{-n+3} (\log\sigma)v \Dt v \Dt \sigma\\
&\notag\geq -c_4\beta^{3}\int \sigma^{-n+1}(\log\sigma)^2\,v^2-\frac{1}{4\beta} \int \sigma^{-n+5} (\Dt v)^2.
\end{align}

Thus from \eqref{final-estimate-Zv} - \eqref{i-11} it follows that for large enough $\beta$ and sufficiently small $R_0$, we have
\begin{align}
\label{an-intermediate-estimate}& \sum_{i=2}^{11} \mathcal I_i\\
\notag&\geq 4 \beta\int \sigma^{-n}((n-2)(-\log \sigma)) (Zv)^{2}+\frac{31}{8} \beta\int \sigma^{-n} (Zv)^{2}+\frac{59}{10} \beta^3 \int \sigma^{-n} (\log\sigma)^2 v^2\\
\notag& -2(n-4)\int \sigma^{-n+2} Zv\cdot\Dt v -\frac{1}{\beta} \int \sigma^{-n+5} (\Dt v)^2-\frac{5}{2} \beta \int \sigma^{-n+2} \langle A \nabla v, \nabla v\rangle.
\end{align}

Our focus now will be to  get rid of  the terms $-2(n-4)\int \sigma^{-n+2} Zv\cdot\Dt v$ and $-\frac{1}{\beta} \int \sigma^{-n+5} (\Dt v)^2$ in \eqref{an-intermediate-estimate} above. We split the term $-2(n-4)\int \sigma^{-n+2} Zv\cdot\Dt v$ into two parts, namely
\begin{align*}
&\mathcal{A}:=-2(n-2)\int \sigma^{-n+2} Z(v)\cdot\Dt v\\
&\mathcal{B}:=4 \int \sigma^{-n+2} Z(v)\cdot\Dt v.
\end{align*}
We handle $\mathcal{A}$ and $\mathcal{B}$ separately. 
\medskip
\\
\paragraph{\underline{Estimate for $\mathcal{A}$}:} Let us couple $\mathcal{A}$ with $\frac{1}{2}\int \sigma^{-n+4} A^2$, where we recall, $$\frac{1}{2}\int \sigma^{-n+4} A^2=\frac{1}{2}\int \sigma^{-n+4} [2 \sigma^{-2}(-\beta \log\sigma)\, Zv-\Dt v]^{2}.$$ In that direction, let us observe
\begin{align}
&\label{estimate-A} \mathcal{A}+\frac{1}{2}\int \sigma^{-n+4} A^2\\
\notag &=-2(n-2)\int \sigma^{-n+2} Z(v)\cdot\Dt v+\frac{1}{2}\int \sigma^{-n+4} [2 \sigma^{-2}(-\beta \log\sigma)\, Zv-\Dt v]^{2}\\
&\notag=-2(n-2)\int \sigma^{-n+2} Z(v)\cdot\Dt v\\
\notag&+\int \sigma^{-n+4} \bigg[\sqrt{2} \sigma^{-2}(-\beta \log\sigma)\, Zv-\frac{1}{\sqrt{2}}\Dt v\bigg]^{2}\\
&\notag=\int \sigma^{-n+4} \bigg[\sqrt{2} \sigma^{-2}(-\beta \log\sigma)\, Zv-\frac{1}{\sqrt{2}}\Dt v +\sqrt{2} (n-2)\frac{Zv}{\sigma^2}\bigg]^{2}\\
\notag &-4 \beta\int \sigma^{-n}((n-2)(-\log \sigma)) (Zv)^{2}-2(n-2)^2\int \sigma^{-n}(Zv)^2\\
&\geq -4 \beta\int \sigma^{-n}((n-2)(-\log \sigma)) (Zv)^{2}-2(n-2)^2\int \sigma^{-n}(Zv)^2\notag.
\end{align}
Here we would like to emphasize  that the negative term appearing above, i.e., $-4 \beta\int \sigma^{-n}((n-2)(-\log \sigma)) (Zv)^{2}$  will be cancelled by the exact positive term present in \eqref{an-intermediate-estimate} and we will eventually choose $\beta$ large enough such that

\begin{equation}\label{vh1}
2(n-2)^2\int \sigma^{-n}(Zv)^2  \leq \frac{\beta}{8} \int \sigma^{-n} (Zv)^2.\end{equation} 

Now we focus on getting rid of  the negative term $-\frac{1}{\beta} \int \sigma^{-n+5} (\Dt v)^2$ present in \eqref{an-intermediate-estimate}. This will be tackled by $\mathcal{B}$ and $\frac{1}{2}\int \sigma^{-n+4} A^2$. To see that, we observe
\paragraph{\underline{Estimate for $\mathcal{B}$}:}
\begin{align}
\label{estimate-B}&\mathcal{B}+\frac{1}{2}\int \sigma^{-n+4} A^2\\
\notag &=4 \int \sigma^{-n+2} Zv\cdot\Dt v+\frac{1}{2}\int \sigma^{-n+4} [2 \sigma^{-2}(-\beta \log\sigma)\, Zv-\Dt v]^{2}\\
\notag &=4 \int \sigma^{-n+2} Zv\cdot\Dt v+\frac{1}{2}\int \sigma^{-n+4} \bigg[2 \sigma^{-2}(-\beta \log\sigma)\, Zv-\left(1-\frac{2}{(-\beta \log\sigma)}\right)\Dt v-\frac{2\Dt v}{\beta (-\log\sigma)}\bigg]^{2}\\
\notag&=4 \int \sigma^{-n+2} Zv\cdot\Dt v+\frac{1}{2}\int \sigma^{-n+4} \bigg[2 \sigma^{-2}(-\beta \log\sigma)\, Zv-\left(1-\frac{2}{(-\beta \log\sigma)}\right)\Dt v\bigg]^{2}\\
&\notag -2 \int \sigma^{-n+4}\bigg[2 \sigma^{-2}(-\beta \log\sigma)\, Zv-\left(1-\frac{2}{(-\beta \log\sigma)}\right)\Dt v\bigg]\frac{\Dt v}{\beta(-\log\sigma)}+2\int \sigma^{-n+4} \frac{(\Dt v)^2}{(\beta \log\sigma)^2}\\
\notag &\geq \int \sigma^{-n+4} \frac{(\Dt v)^2}{(\beta \log\sigma)^2}\left(2+2\beta (-\log\sigma)-4\right)\\
&\geq \frac{1}{\beta} \int \sigma^{-n+5} (\Dt v)^2,\notag
\end{align}
provided we choose $\beta$ large enough and $R_0$ small enough,  From \eqref{an-intermediate-estimate},  \eqref{estimate-A}, \eqref{vh1} and \eqref{estimate-B} we thus have\begin{align}
\label{a-semi-final-estimate} \int \sigma^{-n+4} \, e^{\beta(\log \sigma)^2}(\La
	 u-\Dt u)^2
	\notag&\geq \frac{59}{10} \beta^3 \int \sigma^{-n} (\log\sigma)^2 v^2+\frac{15}{4} \beta\int \sigma^{-n} (Zv)^{2}\\
	&\notag-\frac{5}{2} \beta \int \sigma^{-n+2} \langle A \nabla v, \nabla v\rangle,
	\end{align}
	which finishes the proof of the theorem.
	\end{proof}
Now in order to incorporate the gradient term in our carleman estimate in Theorem \ref{hardy-BGG},  we  will use an interpolation type argument.  As an intermediate step, we first relate  $\int \sigma^{-n+2} \langle A \nabla v, \nabla v\rangle$ with \\ $\int \sigma^{-n+2 } e^{\beta (\log \sigma)^2}\langle A\nabla u, \nabla u \rangle$.
\begin{lemma}
The following holds true for $R_0$ small enough and $\beta$ large.

    \begin{equation}\label{dom1-BGG}
	\int \sigma^{-n+2 } e^{\beta (\log \sigma)^2}\langle A\nabla u, \nabla u \rangle \geq  \int \sigma^{-n+2} \langle A\nabla v, \nabla v \rangle + \frac{4\beta^2}{5} \int \sigma^{-n} (\log \sigma)^2 v^2.  \end{equation}
\
\end{lemma}
\begin{proof}
We have, 
	\begin{align}\label{on0}
		&\int \sigma^{-n+2} e^{\beta (\log \sigma)^2} \langle A\nabla u, \nabla u \rangle= \int \sigma^{-n+2} e^{\beta (\log \sigma)^2}\langle A \nabla e^{-\beta/2 (\log \sigma)^2} v,\nabla e^{-\beta/2 (\log \sigma)^2} v \rangle\\
		&=\int \sigma^{-n+2}  \bigg[\frac{(\beta \log \sigma)^2}{\sigma^2} \, v^2 \langle A \nabla \sigma, \nabla \sigma \rangle+2\frac{(-\beta \log \sigma)}{\sigma^2} v Zv +  \, <A\nabla v, \nabla v> \bigg]. \notag
	\end{align}
	
	We now claim that
	\begin{equation}\label{clm1}
	\langle A \nabla \sigma, \nabla \sigma \rangle= 1 + O(\sigma). \end{equation}
	
	\eqref{clm1} is seen by writing $a_{ij}(x,t)= a_{ij}(0,t) + b_{ij}(x,t)$. We thus have
	\begin{align}\label{on1}
	&\langle A \nabla \sigma, \nabla \sigma \rangle = \frac{a_{ij}(0,t) g_{il}(0,t) x_l g_{jm}(0,t) x_m}{\sigma^2} + b_{ij} (x,t) \partial_i \sigma \partial_j \sigma
	\\
	& = \frac{\delta_{jl} x_l  g_{jm} x_m}{\sigma^2} + O(\sigma)\ \text{(using $|\nabla \sigma|= O(1)$ and $|b_{ij}(x,t)|= O(\sigma)$.)}
	\notag
	\\
	&= \frac{g_{jm} x_j x_m}{\sigma^2} + O(\sigma)\notag\\
	&= 1+ O(\sigma).\notag
	\end{align}
	Also we have
	\begin{align}\label{on2}
	&2\int \sigma^{-n}  (-\beta \log \sigma) vZv = \beta \int \operatorname{div}(\sigma^{-n} Z \log \sigma) v^2 \geq - C\beta \int \sigma^{-n} v^2.
	\end{align}
	Using \eqref{on1} and \eqref{on2} in \eqref{on0}, we find that \eqref{dom1-BGG} follows 	
	provided $\beta$ is sufficiently  large and $R_0$ is small enough. 
\end{proof}
Using \eqref{dom1-BGG} in \eqref{semi-final-BGG} we have
\begin{align}
  \label{almost-final-estimate} &\int \sigma^{-n+4} \, e^{\beta(\log \sigma)^2}(\La
	 u-\Dt u)^2+\frac{5}{2}\beta \int \sigma^{-n+2 } e^{\beta (\log \sigma)^2}\langle A\nabla u, \nabla u \rangle\\
	 \notag & \geq \frac{79}{10} \beta^3 \int \sigma^{-n} (\log\sigma)^2 v^2+\frac{15}{4} \beta\int \sigma^{-n} (Zv)^{2}.
\end{align}
We now state the prove the relevant interpolation lemma that allows us to incorporate the gradient term in our main Carleman estimate.

\begin{lemma}[Interpolation lemma]
The following holds true for small $R_0$ and sufficiently large $\beta$
	\begin{align}\label{grad1-BGG}
	& \beta  \int \sigma^{-n+2} e^{\beta (\log \sigma)^2}  \langle A \nabla u, \nabla u\rangle \leq 	C \int \sigma^{-n+4} e^{\beta (\log \sigma)^2} (\La u-\Dt u)^2.
	\end{align}
\end{lemma}

\begin{proof}
\begin{align}\label{c15-BGG}
	&\beta  \int \sigma ^{-n+2} e^{\beta (\log \sigma )^2}   \langle A \nabla u, \nabla u\rangle 
	=- \beta \int \bigg< \nabla(\sigma ^{-n+2} e^{\beta (\log \sigma )^2} ),  A\nabla(e^{-\beta/2 (\log \sigma )^2}v) \bigg> e^{-\beta/2 (\log \sigma )^2}v \\ & -\beta \int \La(e^{-\beta/2 (\log \sigma )^2}v) \sigma ^{-n+2} e^{\beta/2 (\log \sigma )^2}v.
	\notag
	\\
	\notag &=- \beta \int \bigg< \nabla(\sigma ^{-n+2} e^{\beta (\log \sigma )^2} ),  A\nabla(e^{-\beta/2 (\log \sigma )^2}v) \bigg> e^{-\beta/2 (\log \sigma )^2}v \\ & -\beta \int (\La u-\Dt u) \sigma ^{-n+2} e^{\beta/2 (\log \sigma )^2}v
	\notag-\beta \int \Dt u\, \sigma ^{-n+2} e^{\beta/2 (\log \sigma )^2}v
	\notag\\
	& \leq - \beta \int \bigg< \nabla(\sigma ^{-n+2} e^{\beta (\log \sigma )^2}),  A\nabla(e^{-\beta/2 (\log \sigma )^2}v) \bigg> e^{-\beta/2 (\log \sigma )^2}v\notag\\ & + C \int \sigma ^{-n+4}  e^{\beta (\log \sigma )^2} (\La u-\Dt u)^2 + C\beta^2 \int \sigma ^{-n}   v^2-\beta \int \Dt u\, \sigma ^{-n+2} e^{\beta/2 (\log \sigma )^2}v.
	\notag
	\end{align}
	
	Now by integration by parts in $t$ and by using $|\Dt \sigma|=O(\sigma)$, we find that 
	\begin{align}\label{on5}
	-\beta \int \Dt u\, \sigma ^{-n+2} e^{\beta/2 (\log \sigma )^2}v 	 \leq C_1 \beta^2 \int \sigma^{-n} v^2.
	\end{align}
We thus have
\begin{align}\label{on7}
&\beta  \int \sigma ^{-n+2} e^{\beta (\log \sigma )^2}   \langle A \nabla u, \nabla u\rangle 
\\
& \leq - \beta \int \bigg< \nabla(\sigma ^{-n+2} e^{\beta (\log \sigma )^2}),  A\nabla(e^{-\beta/2 (\log \sigma )^2}v) \bigg> e^{-\beta/2 (\log \sigma )^2}v\notag\\ & + C \int \sigma ^{-n+4}  e^{\beta (\log \sigma )^2} (\La u-\Dt u)^2 + C\beta^2 \int \sigma ^{-n}   v^2.\notag
\end{align}	
Now the term $C\beta^2 \int \sigma^{-n} v^2$ can be estimated using \eqref{almost-final-estimate} and we consequently have
\begin{align}\label{bn01}
	&\beta  \int \sigma ^{-n+2} e^{\beta (\log \sigma )^2}   \langle A \nabla u, \nabla u\rangle   \leq - \beta \int \bigg< \nabla(\sigma ^{-n+2} e^{\beta (\log \sigma )^2}),  A\nabla(e^{-\beta/2 (\log \sigma )^2}v) \bigg> e^{-\beta/2 (\log \sigma )^2}v \\ &+ C \int \sigma ^{-n+4}  e^{\beta (\log \sigma )^2} (\La u-\Dt u)^2 	
	+ C \int \sigma ^{-n+2} e^{\beta (\log \sigma )^2} <A  \nabla u, \nabla u>. \notag
			\end{align}

We are  now left with estimating the term		\[
			\beta \int \bigg< \nabla (\sigma ^{-n+2} e^{\beta (\log \sigma )^2}  ),  A\nabla (e^{-\beta/2 (\log \sigma )^2}v) \bigg> e^{-\beta/2 (\log \sigma )^2}v.\]
			
			This  is done as follows. 
We have  using \eqref{on1}

	\begin{align}\label{c19}
	&\beta \int \langle \nabla (\sigma ^{-n+2} e^{\beta (\log \sigma )^2}  ),  A\nabla(e^{-\beta/2 (\log \sigma )^2}v) \rangle e^{-\beta/2 (\log \sigma )^2}v
	\\
	&=\beta (-n+2) \int \sigma ^{-n}  \left[ \beta(-\log \sigma ) v^2+ \, Zv \cdot v\right]    -\beta \int \sigma ^{-n} 2\beta (-\log \sigma )\left[(-\beta \log \sigma ) v^2+  Zv \cdot v \right]
	\notag
	\\
	& + O(\beta^3) \int \sigma^{-n+1} (\log \sigma)^2 v^2  \geq -\frac{31\beta^{3}}{10}  \int \sigma ^{-n} (\log \sigma )^2  v^2 - \frac{11 \beta}{10} \int \sigma ^{-n} (Zv)^2 \ \text{(for all large $\beta$)}\notag,
	\end{align}
	where in the last inequality we used that
	
	\begin{align}
	&\left| \beta \int \sigma^{-n} 2\beta (-\log \sigma) Zv \cdot v \right| \leq  \beta^3 \int \sigma^{-n} (\log \sigma)^2 v^2  + \beta \int \sigma^{-n} (Zv)^2\ \text{and}\\
	&\left| O(\beta^3) \int \sigma^{-n+1} (\log \sigma)^2 v^2 + \beta  (-n+2) \int \sigma^{-n} \left[ \beta (-\log \sigma) v^2 + Zv \cdot v \right] \right| \notag\\ & \leq \frac{\beta^3}{10} \int \sigma^{-n} (\log \sigma)^2 v^2 + \frac{\beta}{10} \int \sigma^{-n} (Zv)^2\ \text{for small enough $R_0$}.\notag
	\end{align}
	Now   by multiplying the  inequality  \eqref{almost-final-estimate}  on both sides by $\frac{10}{79} \times \frac{31}{10}$, we have
	\begin{align}\label{kjh1}
	&\frac{31}{10} \beta^3 \int \sigma^{-n} (\log \sigma)^2  v^2 + \frac{465 \beta}{316} \int \sigma^{-n }(Zv)^2\\  &\leq \frac{155 \beta}{158} \int \sigma^{-n+2} e^{\beta (\log \sigma)^2}  \langle A \nabla u, \nabla u \rangle +C\int \sigma^{-n+4} \, e^{\beta(\log \sigma)^2}(\La
	 u-\Dt u)^2.\notag	\end{align}
	 
	 Since $\frac{465}{316} > \frac{11}{10}$, we obtain from \eqref{kjh1} that 
	 
	 \begin{align}\label{kjh2}
	&\frac{31}{10} \beta^3 \int \sigma^{-n} (\log \sigma)^2  v^2 + \frac{11\beta}{10} \int \sigma^{-n }(Zv)^2\\  &\leq \frac{155 \beta}{158} \int \sigma^{-n+2} e^{\beta (\log \sigma)^2}  \langle A \nabla u, \nabla u \rangle +C\int \sigma^{-n+4} \, e^{\beta(\log \sigma)^2}(\La
	 u-\Dt u)^2.\notag	\end{align}

	 Using \eqref{c19} and  \eqref{kjh2} in  \eqref{bn01} we have	 
	
	\begin{align}\label{grad}
	& \beta  \int \sigma ^{-n+2} e^{\beta (\log \sigma )^2}  \langle A \nabla u, \nabla u\rangle \leq 	C \int \sigma ^{-n+4} e^{\beta (\log \sigma )^2} (\La u - u_t)^2 \\
	\notag &+ \left( C +\frac{155}{158} \beta\right) \int \sigma ^{-n+2}  e^{\beta (\log \sigma)^2}\langle A\nabla u, \nabla u \rangle.
	\end{align}
	Now for all $\beta$ large  enough, we observe that the  following term in \eqref{grad} above, i.e.
	\[
	\left( C +\frac{155}{158} \beta \right) \int \sigma ^{-n+2} e^{\beta (\log \sigma)^2}  \langle A\nabla u, \nabla u \rangle,\] can be absorbed in the left hand side of \eqref{grad} and we thus infer that the following estimate holds,
	
	\begin{align}\label{grad1-BGG}
	& \beta  \int \sigma ^{-n+2} e^{\beta (\log \sigma )^2}  \langle A \nabla u, \nabla u\rangle \leq 	C \int \sigma ^{-n+4} e^{\beta (\log \sigma )^2} (\La u-\Dt u)^2 .
	\end{align}
	The conclusion thus follows.
	
	\end{proof}

	\begin{proof}[Proof of Theorem \ref{hardy-BGG}]
	The desired estimate \eqref{har1-BGG} now follows from \eqref{almost-final-estimate} and \eqref{grad1-BGG}.
	
	\end{proof}
	
	With Theorem \ref{hardy-BGG} in hand, we now show that  nontrivial solutions to \eqref{main-diff-ineq} decay less than exponentially.

\begin{proof}[Proof of Theorem~\ref{thm1-bgg}]
We follow the strategy as in \cite{BGM, V}. The proof is also partly similar to that of Proposition \ref{expdecay}.
As before,  in what follows  we assume that $u$ solves \eqref{main-diff-ineq} in $B_{R} \times (-T,T)$, instead of $ B_{R}\times(0,T)$. By applying   Proposition \ref{expdecay}, we have  that

\begin{equation}\label{po}
    \int_{B_s\times (-T/2,T/2)}u^2\lesssim e^{-\frac{C}{s^{2/3}}},~\text{as}~s\rightarrow0
\end{equation}
for some constant $C>0$. We first show that $u(\cdot, t) =0$ for $|t|  \leq \frac{T}{2}$.   Then  by repeating the arguments, one can spread the zero set  for $|t| > \frac{T}{2}$  to finally assert that $u \equiv 0$.

In view of \eqref{po}, we  now  let $T/2$ as our new  $T$.
  Without loss of generality, we  assume that $R<1$. Let $\phi(x)$ be a smooth function such that
$$\phi(x)\equiv \begin{cases} 0,\, |x| <\frac{r_{1}}{2};\\
1,\,r_{1}< |x|< r_{2};\\
0,\,|x|>r_{3},
\end{cases}$$
where, $0 < r_1 < r_2/2 < 4r_2 <r_3 < R/2$ will be fixed at a later stage. We subsequently let $T_2= T/2$ and $T_1= 3T/4$, so that $0<T_2<T_1<T$. As before, we let $\eta(t)$ be a smooth even function such that $\eta(t) \equiv 1$ when $|t| < T_2$, $\eta(t) \equiv 0$, when $|t| > T_1$. More precisely
\begin{equation}\label{d4}
\eta(t)= \begin{cases} 0\ \ \ \ \ \ \ \ \ -T\le t\le -T_1
\\
\exp \left(-\frac{T^3(T_2+t)^4}{(T_1 +t)^3(T_1-T_2)^4} \right)\ \ \ \ \ \ \ -T_1\le t \le -T_2,
\\
1,\ \ \ \ \ \ \  -T_2\le t \le 0.
 \end{cases}
 \end{equation}
  Without loss of generality we assume that
 \begin{equation}\label{assume-bgg}
   \int_{ B_{\frac{M r_2}{N}} \times (-T_2, T_2)} u^2 \neq 0,
  \end{equation}
where $M, N$ are constants such that $M|x|\leq \sigma(x,t)\leq N|x|$. Otherwise, the result in \cite{V} implies $u \equiv 0$  in $B_R \times (-T_2, T_2)$ and by arguments that follow, we could conclude that $u \equiv 0$ also  for $|t| > T_2$. 

Now, with $u$ as in Theorem \ref{thm1-bgg} we let  $v= \phi \eta u$. Limiting argument allows to use such $v$ in the Carleman estimate \eqref{hardy-BGG}, obtaining
\begin{align}\label{thm-bgg-carleman}
	&\beta^3 \int \sigma^{-n} e^{\beta (\log \sigma)^2}v^2 +  \beta \int  \sigma^{-n+2} e^{\beta (\log \sigma)^2} \langle A \, \nabla v, \nabla v\rangle\\
	&\leq C  \, \int \sigma^{-n+4} e^{\beta (\log \sigma)^2} (\La  v-\Dt v)^2. \notag
	\end{align}
Incorporating the following
\begin{align*}
\La v-\Dt v=\phi \eta (\La u - u_t)+ 2\eta \langle A\nabla\phi, \nabla u\rangle+u(\eta \operatorname{div}(A \nabla \phi)-\phi \eta_{t}),
\end{align*}
and also by using \eqref{assump}, we obtain from \eqref{thm-bgg-carleman}
\begin{align}
\notag	&\beta^3 \int \sigma(x, t)^{-n} e^{\beta (\log \sigma(x, t))^2}v^2 +  \beta \int   \sigma(x, t)^{-n+2} e^{\beta (\log \sigma(x, t))^2} \langle A \, \nabla v, \nabla v\rangle\\
	&\leq C \bigg[ \, \int \sigma(x, t)^{-n+4} e^{\beta (\log \sigma(x, t))^2} \phi^2 \eta^2 (\La u - u_t)^2\notag\\
	&+ \int \eta^2\bigg[|\nabla u|^2 |\nabla \phi|^2+u^2 |\nabla^2 \phi|^2+u^2 |\nabla \phi|^2\bigg]\sigma(x, t)^{-n+4} e^{\beta (\log \sigma(x, t))^2}\notag\\
	&+\int u^2 \phi^2 \eta_{t}^2 \sigma(x, t)^{-n+4} e^{\beta(\log\sigma(x, t))^2} \bigg]\notag\\
&\leq C\bigg[  \, \int \sigma(x, t)^{-n} e^{\beta (\log \sigma(x, t))^2} \phi^2 \eta^2 u^2\notag\\
	&+ \int \eta^2\bigg[|\nabla u|^2 |\nabla \phi|^2+u^2 |\nabla^2 \phi|^2+u^2 |\nabla \phi|^2\bigg]\sigma(x, t)^{-n+4} e^{\beta (\log \sigma(x, t))^2}\notag\\
	&+ \int u^2 \phi^2 \eta_{t}^2 \sigma(x, t)^{-n+4} e^{\beta(\log\sigma(x, t))^2}\bigg],
	\end{align}
where in the final inequality we have used the differential inequality \eqref{main-diff-ineq}.	Therefore, for sufficiently large $\beta$,  we find that the term $C  \, \int \sigma(x, t)^{-n} e^{\beta (\log \sigma(x, t))^2} \phi^2 \eta^2 u^2$ can be absorbed in the left hand side and we thus obtain\begin{align}
&\beta^3 \int \sigma(x, t)^{-n} e^{\beta (\log \sigma(x, t))^2}v^2 +  \beta \int  \sigma(x, t)^{-n+2} e^{\beta (\log \sigma(x, t))^2} \langle A \, \nabla v, \nabla v\rangle\\
&\lesssim\notag \int _{ \{r_1/2 < r < r_1\}\times (-T_1,T_1) }\eta^2 \bigg[|\nabla u|^2 \sigma(x, t)^{-n+2}+u^2 \sigma(x, t)^{-n} \bigg] e^{\beta (\log \sigma(x, t))^2}\\
&+\notag \int _{ \{r_2 < r < r_3\}\times (-T_1,T_1) }\eta^2 \bigg[|\nabla u|^2 \sigma(x, t)^{-n+2}+u^2 \sigma(x, t)^{-n} \bigg] e^{\beta (\log \sigma(x, t))^2}\\
&+\int u^2 \phi^2 \eta_{t}^2 \sigma(x, t)^{-n+4} e^{\beta(\log\sigma(x, t))^2}.\notag
\end{align}
The above inequality can be rewritten as follows
\begin{align}
\label{carleman-2}
&\frac{\beta^3}{2} \int \sigma(x, t)^{-n} e^{\beta (\log \sigma(x, t))^2}v^2 +  \beta \int \sigma(x, t)^{-n+2} e^{\beta (\log \sigma(x, t))^2} \langle A \, \nabla v, \nabla v\rangle\\ 
&\lesssim\notag \int _{ \{r_1/2 < r < r_1\}\times (-T_1,T_1) }\eta^2 \bigg[|\nabla u|^2 \sigma(x, t)^{-n+2}+u^2 \sigma(x, t)^{-n} \bigg] e^{\beta (\log \sigma(x, t))^2}\\
&+\notag \int _{ \{r_2 < r < r_3\}\times (-T_1,T_1) }\eta^2 \bigg[|\nabla u|^2 \sigma(x, t)^{-n+2}+u^2 \sigma(x, t)^{-n} \bigg] e^{\beta (\log \sigma(x, t))^2}\\
\notag &+\int u^2 \phi^2 \eta_{t}^2 \sigma(x, t)^{-n+4} e^{\beta(\log\sigma(x, t))^2}-\frac{\beta^3}{2} \int \sigma(x, t)^{-n} e^{\beta (\log \sigma(x, t))^2}v^2:=I_1+I_2+I_3+I_4.
\end{align}
Recall the following notations introduced in \eqref{sigmas}
\begin{align}
    \sigma_1(r_1, r_2, T):=\inf_{\{r_1<|x|<r_2\}\times (-T, T)}\sigma(x,t),~\text{and}~ \sigma_2(r_1, r_2, T):=\sup_{\{r_1<|x|<r_2\}\times (-T, T)}\sigma(x,t).
\end{align}
\medskip

\paragraph{\underline{\textbf{Estimate for $I_1$}}}
\begin{align}
\notag I_1&=\int _{ \{r_1/2 < r < r_1\}\times (-T_1,T_1) }\eta^2 \bigg[|\nabla u|^2 \sigma(x, t)^{-n+2}+u^2 \sigma(x, t)^{-n} \bigg] e^{\beta (\log \sigma(x, t))^2}\\
\notag &\lesssim \sigma_{1}(r_1/2, r_1, T)^{-n}\eer\int_{\{r_1/2<|x|<r_1\}\times (-T_1, T_1)} u^2\\
&+\sigma_{1}(r_1/2, r_1, T)^{-n+2}\eer\int_{\{r_1/2<|x|<r_1\}\times (-T_1, T_1)} |\nabla u|^2\notag\\
&\lesssim \sigma_{1}(r_1/2, r_1, T)^{-n}\eer\int_{\{r_1/4<|x|<3{r_1}/2\}\times (-T_1, T_1)} |u|^2,
\label{expo-growth-1}
\end{align}
where in the last inequality above,  we  used  the energy estimate from Lemma~\ref{cac-bgg}. Similarly

\medskip

\paragraph{\underline{\textbf{Estimate for $I_2$}}}
\begin{align}
\notag I_2&=\int _{ \{r_2 < r < r_3\}\times (-T_1,T_1) }\eta^2 \bigg[|\nabla u|^2 \sigma^{-n+2}+u^2 \sigma^{-n} \bigg] e^{\beta (\log \sigma)^2}\\
&\lesssim \sigma_{1}(r_2, r_3, T)^{-n}\ees \int_{B_R\times (-T, T)} u^2.
\label{expo-growth-2}
\end{align}

\medskip

Let us decompose $I_3$ into several parts.
\begin{align}
\notag I_3 &:=\int u^2 \phi^2 \eta_{t}^2 \sigma(x, t)^{-n+4} e^{\beta(\log\sigma(x, t))^2}\\
\notag &:=\int_{\{r_1/2 < |x| < r_1\} \times (-T_1,T_1)} u^2 \phi^2 \eta_{t}^2 \sigma(x, t)^{-n+4} e^{\beta(\log\sigma(x, t))^2}\\
\notag &+\int_{\{r_1 < |x| < r_2\} \times (-T_1,T_1)} u^2 \phi^2 \eta_{t}^2 \sigma(x, t)^{-n+4} e^{\beta(\log\sigma(x, t))^2}\\
\notag &+\int_{\{r_2 < |x| < r_3\} \times (-T_1, T_1)} u^2 \phi^2 \eta_{t}^2 \sigma(x, t)^{-n+4} e^{\beta(\log\sigma(x, t))^2}\\
&\notag\lesssim \sigma_{1}(r_1/2, r_1, T)^{-n}\eer\int_{\{r_1/2 < |x| < r_1\} \times (-T_1,T_1)} u^2\\
\notag &+\sigma_{1}(r_2, r_3, T)^{-n}\ees\int_{\{r_2 < |x| < r_3\}\times (-T_1, T_1)} u^2\\
&+\int_{\{r_1 < |x| < r_2\} \times (-T_1,T_1)} u^2 \eta_{t}^2 \sigma(x, t)^{-n+4} e^{\beta(\log\sigma(x, t))^2}\label{expo-growth-3}=:I_{31}+I_{32}+I_{33}.
\end{align}
At this point, let us club the term $I_{33}$ with $I_4$. Since, the function $\eta_t$ is supported in the set $(-T_1,-T_2) \cup (T_2, T_1)$, if we denote $\Omega = \{r_1 < |x| < r_2\}\times [(-T_1, -T_2) \cup (T_2, T_1) ]$, we can bound
\begin{align}
\notag I_{33}+I_4:=&\int_{\{r_1 < |x| < r_2\}\times (-T_1, T_1)} u^2 \eta_{t}^2 \sigma(x, t)^{-n+4} e^{\beta(\log\sigma(x, t))^2}-\frac{\beta^3}{2} \int_{B_R} \sigma(x, t)^{-n} e^{\beta (\log \sigma(x, t))^2}v^2\\
& \leq \int_{\Omega} \sigma(x, t)^{-n} e^{\beta (\log \sigma(x, t))^2} u^2 \eta^2 \left(C \sigma(x, t)^{4}   \frac{\eta_t^2}{\eta^2} - \frac{\beta^3}{2}\right)\notag\\
& \lesssim \int_{\Omega} \sigma(x, t)^{-n} e^{\beta (\log \sigma(x, t))^2} u^2 \eta^2 \left(C \sigma(x, t)^{3}   \frac{\eta_t^2}{\eta^2} - \frac{\beta^3}{2}\right)\label{eta-t},
\end{align}
where the last inequality relies on the fact that $\sigma$ is very small. The rest of the proof is devoted in proving the following claim.
\begin{align}
\text{\textbf{Claim:}} \,\, \int_{\Omega} \sigma(x, t)^{-n} e^{\beta (\log \sigma(x, t))^2} u^2 \eta^2 \left(C \sigma(x, t)^{3}   \frac{\eta_t^2}{\eta^2} - \frac{\beta^3}{2}\right)\leq C \int_{B_{R}\times(-T, T)} u^2.  
\label{thm-claim}
\end{align}
It is sufficient to prove \eqref{thm-claim}  over the region $\Omega^{-} = \{r_1 < |x| < r_2\} \times (-T_1, -T_2)$, since the other part can be handled by symmetry. Now, if $-T_1\le t\le -T_2$, note that $T_1 -T_2 = \frac T4$, $|T_2 + t| \le T_1 - T_2 = \frac T4$, and that $\frac 34 T \le 4T_1- 3T_2+t \le T$, it follows similarly as before in the proof of Proposition \ref{expdecay} that
\begin{equation*}\label{y2}
\bigg|\frac{\eta_t}{\eta}\bigg| =\bigg|\frac{T^3 (T_2 + t)^3 (4T_1- 3T_2+t)}{ (T_1-T_2)^4 (T_1 +t)^4}\bigg| \leq   \frac{4 T^3}{|T_1 +t|^4}.
\end{equation*}
Consequently we  obtain
\begin{align*}
&\int_{\Omega^-} \sigma(x, t)^{-n} e^{\beta (\log \sigma(x, t))^2} u^2 \eta^2 \left(C \sigma(x, t)^{3}   \frac{\eta_t^2}{\eta^2} - \frac{\beta^3}{2}\right) 
 \leq  \int_{\Omega^-} \sigma(x, t)^{-n} e^{\beta (\log \sigma(x, t))^2} u^2 \eta^2 \left(C \sigma(x, t)^3  \frac{T^6}{(T_1 +t)^8}  - \frac{\beta^3}{2}\right).
\end{align*}
Trivially
\begin{align}
\label{eta-t-triv} 
\notag&\int_{\Omega^-} \sigma(x, t)^{-n} e^{\beta (\log \sigma(x, t))^2} u^2 \eta^2 \left(C \sigma(x, t)^{3}   \frac{\eta_t^2}{\eta^2} - \frac{\beta^3}{2}\right)\\
&\leq \int_{U} \sigma(x, t)^{-n} e^{\beta (\log \sigma(x, t))^2} u^2 \eta^2 \left(C \sigma(x, t)^3  \frac{T^6}{(T_1 +t)^8}  - \frac{\beta^3}{2}\right),
\end{align}
where
\begin{equation}
    \label{eta-t-useful}
U:=\bigg\{(x,t)\in \Omega^-: \frac{\beta^3}{2}\leq C \sigma(x, t)^3 \frac{T^6}{(T_1+t)^8}\bigg\}.     
\end{equation}
We thus have
\begin{align}\label{c7}
&\int_{\Omega^-} \sigma(x, t)^{-n} e^{\beta (\log \sigma(x, t))^2} u^2 \eta^2 \left(C \sigma(x, t)^{3}   \frac{\eta_t^2}{\eta^2} - \frac{\beta^3}{2}\right)\\
& \leq C \int_{U} \sigma(x, t)^{-n} e^{\beta (\log \sigma(x, t))^2} u^2 \eta \frac{\eta T^6}{(T_1 +t)^8}.\notag
\end{align}
The claim will be achieved if we establish a bound from above of the quantity $\sigma(x, t)^{-n} e^{\beta (\log \sigma(x, t))^2} \eta \frac{\eta T^6}{(T_1 +t)^8}$ in $U$. Appealing to the exponential decay of $\eta$, at $t = - T_1$, see \eqref{d4}, we obtain for $t \in (-T_1, -T_2)$ 
\begin{equation}\label{bn1}
\frac{\eta T^6}{(T_1+ t)^8} \leq 	C.
\end{equation}
Now  we note that the following equivalence holds\begin{align}
\notag &\sigma(x, t)^{-n} e^{\beta (\log \sigma(x, t))^2} \eta\leq 1\\
\notag &\iff n\log\sigma(x, t)-\beta (\log\sigma(x, t))^2-\log\eta\geq 0\\
\label{main-eq} &\iff n\log\sigma(x, t)-\beta (\log\sigma(x, t))^2+\frac{T^3(T_2+t)^4}{(T_1 +t)^3(T_1-T_2)^4}\geq 0.
\end{align}
The following holds true in $U$
\begin{equation*}
\frac{T_1 +t}{T} \leq \left(\frac{C}{2T^2}\right)^{1/8}\left(\frac{\sigma(x, t)}{\beta}\right)^{3/8}
 = C \left(\frac{\sigma(x, t)}{\beta}\right)^{3/8},
\end{equation*}
for some universal $C>0$. For sufficiently large  $\beta$ we have 
\[
C \left(\frac{\sigma(x, t)}{\beta}\right)^{3/8} \le C \left(\frac{R}{\beta}\right)^{3/8} \le \frac1{12}. 
\]Combining the above, we have
\begin{equation}\label{p1-bgg}
\frac{T_1 +t }{T} \leq \frac1{12},
\end{equation}
in $U$, provided $\beta$ is large enough. Also, $\frac T4 = T_1 - T_2 = T_1 + t + |T_2 + t|$, from \eqref{p1-bgg} we conclude that we must have in $U$
\[
|T_2 + t| \geq \frac{T}{6}.
\]
Invoking these information, we have
\begin{align}
\notag &n\log\sigma(x, t)-\beta (\log\sigma(x, t))^2+\frac{T^3(T_2+t)^4}{(T_1 +t)^3(T_1-T_2)^4}\\
\notag &\geq n\log\sigma(x, t)-\beta (\log\sigma(x, t))^2+\left(\frac{4}6\right)^{4} \left(\frac 2C\right)^{3/8} T^{3/4} \left(\frac{\beta}{\sigma(x, t)}\right)^{9/8}\\
&=\left(\frac{4}6\right)^{4} \left(\frac 2C\right)^{3/8} T^{3/4} \left(\frac{\beta}{\sigma(x, t)}\right)^{9/8}-n\log\frac{1}{\sigma(x, t)}-\beta \left(\log\frac{1}{\sigma(x, t)}\right)^2\geq 0,
\label{log-condn}
\end{align}
if  $\sigma$  is sufficiently small, and $\beta$ very large. We highlight the crucial role of the exponent $\beta^{9/8}$, which dominates the linear term in $\beta$ in order to achieve the required conclusion and also for small $\sigma$, $\frac{1}{\sigma^{9/8}}$ overpowers $\log\frac{1}{\sigma}$, as well as $(\log\frac{1}{\sigma})^2$. Therefore for large enough $\beta$ and small enough $R$, we find that \eqref{main-eq} is valid  and  consequently  the claim \eqref{thm-claim} holds. Now combining \eqref{expo-growth-1}, \eqref{expo-growth-2}, \eqref{expo-growth-3} and \eqref{thm-claim}, we obtain
\begin{align}
\label{carleman-3}
&\frac{\beta^3}{2} \int \sigma(x, t)^{-n} e^{\beta (\log \sigma(x, t))^2}v^2\\
&\notag\lesssim \sigma_{1}(r_1/2, r_1, T)^{-n}\eer\int_{\{r_1/4<|x|<3{r_1}/2\}\times (-T_1, T_1)} |u|^2\\
\notag &+\sigma_1(r_2, r_3, T)^{-n}\ees \int_{B_R\times (-T, T)} u^2\\
\notag &+C \int_{B_{R}\times(-T, T)} u^2.
\end{align}
Recall that there exist constants $M$ and $N$, depending on the ellipticity of the coefficient matrix $A$, such that 
\begin{align}
\label{comparable-sigma}
M|x|\leq \sigma(x,t)\leq N|x|.    
\end{align}
The integral in the left-hand side of \eqref{carleman-3} can be bounded from below in the following way
\begin{equation}\label{b5-bgg}
\frac{\beta^3}{2} \int \sigma(x, t)^{-n} e^{\beta (\log \sigma(x, t))^2}v^2\geq \frac{\beta^3}{2}\sigma_2(r_1, {Mr_2}/{N}, T)^{-n}e^{\beta (\log(\sigma_2(r_1, {Mr_2}/{N}, T))^2}\int_{\{r_1 < |x|< \frac{M r_2}{N} \} \times (-T_2,T_2)} u^2. 
\end{equation}
Substituting \eqref{b5-bgg} in \eqref{carleman-3}, and dividing both sides by $\sigma_2(r_1, {Mr_2}/{N}, T)^{-n}e^{\beta (\log(\sigma_2(r_1, {Mr_2}/{N}, T))^2}$, we obtain
\begin{align}
\label{carleman-4}&\frac{\beta^3}{2}\int_{\{r_1 < |x|< \frac{M r_2}{N} \} \times (-T_2,T_2)} u^2\\
\notag&\leq C \left(\frac{\sigma_2(r_1, Mr_2/N, T)}{\sigma_1(r_1/2, r_1, T)}\right)^n \eesr \int _{\{{r_1}/{4}< |x|< {3r_1}/{2}\}\times (-T_1, T_1)} u^2\\
\notag &+\left(\frac{\sigma_2(r_1, Mr_2/N, T)}{\sigma_1(r_2, r_3, T)}\right)^n \eelog \int _{B_{R}\times (-T, T)} u^2\\
\notag &+C\int _{B_{R}\times (-T, T)} u^2.
\end{align}

Keeping in mind \eqref{comparable-sigma}, it is easily seen that

\[
\eelog \leq e^{\beta\bigg[(\log Mr_2)^2-(\log Mr_2)^2\bigg]}\leq 1.\]

 Using this in \eqref{carleman-4} we deduce the following estimate
\begin{align}
\label{carleman-extra}
&\frac{\beta^3}{2}\int_{\{r_1 < |x|< \frac{M r_2}{N} \} \times (-T_2,T_2)} u^2\\
\notag &\lesssim_{M, N} \left(\frac{2 r_2}{r_1}\right)^n e^{\beta (\log c\,r_1)^2}\int _{\{{r_1}/{4}< |x|< {3r_1}/{2}\}\times (-T_1, T_1)} u^2\\
\notag &+ \tilde{C}\int _{B_{R}\times(-T, T)} u^2.
\end{align}

Adding $\frac{\beta^3}{2}\int_{B_{r_{1}}\times (-T_2, T_2)} u^2$ and choosing $R$ small enough such that $e^{\beta (\log\frac{r_1}{2})^2}\geq \frac{\beta^3}{2}$, we obtain
\begin{align}
\label{carleman-5}
\frac{\beta^3}{2}\int_{B_{\frac{M r_2}{N}} \times (-T_2,T_2)} u^2\leq & C \left(\frac{2 r_2}{r_1}\right)^n e^{\beta (\log c\,r_1)^2}\int _{B_{\frac{3r_1}{2}}\times (-T_1, T_1)} u^2\\
\notag &+\tilde{C}\int _{B_{R}\times (-T, T)} u^2.
\end{align}
Keeping in mind the assumption  \eqref{assume-bgg}, we can choose $\beta$ large enough such that
\begin{align}
\label{condn-3}
\frac{\beta^3}{8}\int_{B_{\frac{M r_2}{N}} \times (-T_2,T_2)} u^2\geq \tilde{C}\int _{B_{R}\times (-T, T)} u^2,    
\end{align}
where $\tilde{C}$ is the constant appearing in \eqref{carleman-5}. Subtracting off $\tilde{C}\int _{B_{R}\times (-T, T)} u^2$ from both sides of \eqref{carleman-5} reduces to following
\begin{align}
\label{carleman-6}    
e^{-\beta (\log\frac{c}{r_1})^2} \left(\frac{r_1}{2r_2}\right)^n\frac{\beta^3}{8}\int_{B_{\frac{M r_2}{N}} \times (-T_2,T_2)} u^2\leq \int _{B_{\frac{3r_1}{2}}\times (-T_1, T_1)} u^2.
\end{align}
Now we fix some $\beta$, sufficiently large, for which \eqref{p1-bgg}, \eqref{log-condn}, and \eqref{condn-3} hold simultaneously. Therefore, \eqref{carleman-6} implies that for sufficiently small $s$, there is a constant $\kappa$, depending on $r_2, r_3,R, \beta$ and the ratio $\frac{\int_{  B_R \times (-T, T)  } u^2}{\int_{ B_{\frac{Mr_2}{N}} \times (-T_2,T_2)} u^2}$ such that
\begin{align*}
\int_{B_{s}\times (-T, T)}u^2\geq C e^{-\kappa(\log\frac{2}{s})^2}.
\end{align*}
This contradicts \eqref{po}. Thus we conclude that $u\equiv 0$.
   \end{proof}

\end{document}